\documentclass[12pt]{amsart}
 \usepackage[margin=1in]{geometry}
 \usepackage{amsmath,amsfonts,amsthm,amssymb,bbm}
 \usepackage{graphicx,color,dsfont}
 \usepackage{enumitem}
 \usepackage{fourier}

\begin{document}

\newtheorem{theorem}{Theorem}
\newtheorem{lemma}{Lemma}
\newtheorem{proposition}{Proposition}
\newtheorem{rmk}{Remark}
\newtheorem{example}{Example}
\newtheorem{exercise}{Exercise}
\newtheorem{definition}{Definition}
\newtheorem{corollary}{Corollary}
\newtheorem{notation}{Notation}
\newtheorem{claim}{Claim}

\newtheorem{dif}{Definition}

\newtheorem{thm}{Theorem}[section]
\newtheorem{cor}[thm]{Corollary}
\newtheorem{lem}[thm]{Lemma}
\newtheorem{prop}[thm]{Proposition}
\theoremstyle{definition}
\newtheorem{defn}[thm]{Definition}
\theoremstyle{remark}
\newtheorem{rem}[thm]{Remark}
\newtheorem*{ex}{Example}
\numberwithin{equation}{section}

\newcommand{\vertiii}[1]{{\left\vert\kern-0.25ex\left\vert\kern-0.25ex\left\vert #1
			\right\vert\kern-0.25ex\right\vert\kern-0.25ex\right\vert}}

\newcommand{\R}{{\mathbb R}}
\newcommand{\C}{{\mathbb C}}
\newcommand{\U}{{\mathcal U}}
\newcommand{\norm}[1]{\left\|#1\right\|}
\renewcommand{\(}{\left(}
\renewcommand{\)}{\right)}
\renewcommand{\[}{\left[}
\renewcommand{\]}{\right]}
\newcommand{\f}[2]{\frac{#1}{#2}}
\newcommand{\im}{i}
\newcommand{\cl}{{\mathcal L}}
\newcommand{\ck}{{\mathcal K}}

%GREEK LETTERS
\newcommand{\al}{\alpha}
\newcommand{\be}{\beta}
\newcommand{\wh}[1]{\widehat{#1}}
\newcommand{\ga}{\gamma}
\newcommand{\Ga}{\Gamma}
\newcommand{\de}{\delta}
\newcommand{\ben}{\beta_n}
\newcommand{\De}{\Delta}
\newcommand{\ve}{\varepsilon}
\newcommand{\ze}{\zeta}
\newcommand{\Th}{\Theta}
\newcommand{\ka}{\kappa}
\newcommand{\la}{\lambda}
\newcommand{\laj}{\lambda_j}
\newcommand{\lak}{\lambda_k}
\newcommand{\La}{\Lambda}
\newcommand{\si}{\sigma}
\newcommand{\Si}{\Sigma}
\newcommand{\vp}{\varphi}
\newcommand{\om}{\omega}
\newcommand{\Om}{\Omega}
\newcommand{\ra}{\rightarrow}

%EUCLIDEAN SPACES
\newcommand{\ro}{{\mathbf R}}
\newcommand{\rn}{{\mathbf R}^n}
\newcommand{\rd}{{\mathbf R}^d}
\newcommand{\rmm}{{\mathbf R}^m}
\newcommand{\rone}{\mathbb R}
\newcommand{\rtwo}{\mathbf R^2}
\newcommand{\rthree}{\mathbf R^3}
\newcommand{\rfour}{\mathbf R^4}
\newcommand{\ronen}{{\mathbf R}^{n+1}}
\newcommand{\ku}{\mathbf u}
\newcommand{\kw}{\mathbf w}
\newcommand{\kf}{\mathbf f}
\newcommand{\kz}{\mathbf z}

\newcommand{\N}{\mathbf N}

\newcommand{\tn}{\mathbf T^n}
\newcommand{\tone}{\mathbf T^1}
\newcommand{\ttwo}{\mathbf T^2}
\newcommand{\tthree}{\mathbf T^3}
\newcommand{\tfour}{\mathbf T^4}

\newcommand{\zn}{\mathbf Z^n}
\newcommand{\zp}{\mathbf Z^+}
\newcommand{\zone}{\mathbf Z^1}
\newcommand{\zz}{\mathbf Z}
\newcommand{\ztwo}{\mathbf Z^2}
\newcommand{\zthree}{\mathbf Z^3}
\newcommand{\zfour}{\mathbf Z^4}

\newcommand{\hn}{\mathbf H^n}
\newcommand{\hone}{\mathbf H^1}
\newcommand{\htwo}{\mathbf H^2}
\newcommand{\hthree}{\mathbf H^3}
\newcommand{\hfour}{\mathbf H^4}

\newcommand{\cone}{\mathbf C^1}
\newcommand{\ctwo}{\mathbf C^2}
\newcommand{\cthree}{\mathbf C^3}
\newcommand{\cfour}{\mathbf C^4}
\newcommand{\dpr}[2]{\langle #1,#2 \rangle}

\newcommand{\sn}{\mathbf S^{n-1}}
\newcommand{\sone}{\mathbf S^1}
\newcommand{\stwo}{\mathbf S^2}
\newcommand{\sthree}{\mathbf S^3}
\newcommand{\sfour}{\mathbf S^4}

%SPACES
\newcommand{\lp}{L^{p}}
\newcommand{\lppr}{L^{p'}}
\newcommand{\lqq}{L^{q}}
\newcommand{\lr}{L^{r}}
\newcommand{\echi}{(1-\chi(x/M))}
\newcommand{\chip}{\chi'(x/M)}

\newcommand{\wlp}{L^{p,\infty}}
\newcommand{\wlq}{L^{q,\infty}}
\newcommand{\wlr}{L^{r,\infty}}
\newcommand{\wlo}{L^{1,\infty}}

\newcommand{\lprn}{L^{p}(\rn)}
\newcommand{\lptn}{L^{p}(\tn)}
\newcommand{\lpzn}{L^{p}(\zn)}
\newcommand{\lpcn}{L^{p}(\cn)}
\newcommand{\lphn}{L^{p}(\cn)}

\newcommand{\lprone}{L^{p}(\rone)}
\newcommand{\lptone}{L^{p}(\tone)}
\newcommand{\lpzone}{L^{p}(\zone)}
\newcommand{\lpcone}{L^{p}(\cone)}
\newcommand{\lphone}{L^{p}(\hone)}

\newcommand{\lqrn}{L^{q}(\rn)}
\newcommand{\lqtn}{L^{q}(\tn)}
\newcommand{\lqzn}{L^{q}(\zn)}
\newcommand{\lqcn}{L^{q}(\cn)}
\newcommand{\lqhn}{L^{q}(\hn)}

\newcommand{\lo}{L^{1}}
\newcommand{\lt}{L^{2}}
\newcommand{\li}{L^{\infty}}
\newcommand{\beqn}{\begin{eqnarray*}}
		\newcommand{\eeqn}{\end{eqnarray*}}
\newcommand{\pplus}{P_{Ker[\cl_+]^\perp}}

\newcommand{\co}{C^{1}}
\newcommand{\coi}{C_0^{\infty}}

%CALLIGRAPHIC
\newcommand{\ca}{\mathcal A}
\newcommand{\cs}{\mathcal S}
\newcommand{\cg}{\mathcal G}
\newcommand{\cm}{\mathcal M}
\newcommand{\cf}{\mathcal F}
\newcommand{\cb}{\mathcal B}
\newcommand{\ce}{\mathcal E}
\newcommand{\cd}{\mathcal D}
\newcommand{\cn}{\mathcal N}
\newcommand{\cz}{\mathcal Z}
\newcommand{\crr}{\mathbf R}
\newcommand{\cc}{\mathcal C}
\newcommand{\ch}{\mathcal H}
\newcommand{\cq}{\mathcal Q}
\newcommand{\cp}{\mathcal P}
\newcommand{\cx}{\mathcal X}
\newcommand{\ci}{\mathcal I}
\newcommand{\cj}{\mathcal J}
\newcommand{\eps}{\epsilon}

%GENERAL
\newcommand{\pv}{\textup{p.v.}\,}
\newcommand{\loc}{\textup{loc}}
\newcommand{\intl}{\int\limits}
\newcommand{\iintl}{\iint\limits}
\newcommand{\dint}{\displaystyle\int}
\newcommand{\diint}{\displaystyle\iint}
\newcommand{\dintl}{\displaystyle\intl}
\newcommand{\diintl}{\displaystyle\iintl}
\newcommand{\liml}{\lim\limits}
\newcommand{\suml}{\sum\limits}
\newcommand{\ltwo}{L^{2}}
\newcommand{\supl}{\sup\limits}
\newcommand{\df}{\displaystyle\frac}
\newcommand{\p}{\partial}
\newcommand{\Ar}{\textup{Arg}}
\newcommand{\abssigk}{\widehat{|\si_k|}}
\newcommand{\ed}{(1-\p_x^2)^{-1}}
\newcommand{\tT}{\tilde{T}}
\newcommand{\tV}{\tilde{V}}
\newcommand{\wt}{\widetilde}
\newcommand{\Qvi}{Q_{\nu,i}}
\newcommand{\sjv}{a_{j,\nu}}
\newcommand{\sj}{a_j}
\newcommand{\pvs}{P_\nu^s}
\newcommand{\pva}{P_1^s}
\newcommand{\cjk}{c_{j,k}^{m,s}}
\newcommand{\Bjsnu}{B_{j-s,\nu}}
\newcommand{\Bjs}{B_{j-s}}
\newcommand{\Ly}{L_i^y}
\newcommand{\dd}[1]{\f{\partial}{\partial #1}}
\newcommand{\czz}{Calder\'on-Zygmund}
\newcommand{\chh}{\mathcal H}

%BEGIN END
\newcommand{\lbl}{\label}
\newcommand{\beq}{\begin{equation}}
		\newcommand{\eeq}{\end{equation}}
\newcommand{\beqna}{\begin{eqnarray*}}
		\newcommand{\eeqna}{\end{eqnarray*}}
\newcommand{\bp}{\begin{proof}}
		\newcommand{\ep}{\end{proof}}
\newcommand{\bprop}{\begin{proposition}}
		\newcommand{\eprop}{\end{proposition}}
\newcommand{\bt}{\begin{theorem}}
		\newcommand{\et}{\end{theorem}}
\newcommand{\bex}{\begin{Example}}
		\newcommand{\eex}{\end{Example}}
\newcommand{\bc}{\begin{corollary}}
		\newcommand{\ec}{\end{corollary}}
\newcommand{\bcl}{\begin{claim}}
		\newcommand{\ecl}{\end{claim}}
\newcommand{\bl}{\begin{lemma}}
		\newcommand{\el}{\end{lemma}}
\newcommand{\dea}{(-\De)^\be}
\newcommand{\naa}{|\nabla|^\be}

\title[Ground states for the Ostrovskyi equation ]
{On the ground states of the   Ostrovskyi equation and their stability}

\author[Iurii  Posukhovskyi]{Iurii Posukhovskyi }
\address{ Department of Mathematics,
	University of Kansas,
	1460 Jayhawk Boulevard,  Lawrence KS 66045--7523, USA}

\email{i.posukhovskyi@ku.edu}

\thanks{ Posukhovskyi is partially supported from a graduate fellowship by NSF-DMS under grant  \# 1614734.   Stefanov    is partially  supported by  NSF-DMS under grants  \# 1614734 and \# 1908626.}

%----------Author 2
\author[Atanas Stefanov]{\sc Atanas Stefanov}
\address{ Department of Mathematics,
	University of Kansas,
	1460 Jayhawk Boulevard,  Lawrence KS 66045--7523, USA}
\email{stefanov@ku.edu}



\date{\today}

\begin{abstract}
	The  Ostrovskyi (Ostrovskyi-Vakhnenko/short pulse)  equations are ubiquitous models  in mathematical physics. They describe water waves under the action of a Coriolis force as well as the amplitude of a ``short'' pulse in an optical fiber.

	In this paper, we rigorously construct  ground  traveling waves for  these models as  minimizers of  the Hamiltonian functional for any  fixed $L^2$ norm. The existence argument proceeds via the method of compensated compactness, but it requires  surprisingly detailed  Fourier analysis arguments to rule out the non-vanishing  of the limits of the minimizing sequences. We show that  all  of these  waves are  weakly non-degenerate and  spectrally stable. 
\end{abstract}

%%% ----------------------------------------------------------------------
\maketitle
%%% ----------------------------------------------------------------------
%\tableofcontents
\section{ Introduction and main results}

The Ostrovskyi model, which is ubiquitous in the modern water waves theory,    is given by, 
\begin{equation}
	\label{20}
	(u_t - u_{xxx}-(u^2)_x)_x=u.
\end{equation}
The related, generalized Ostrovskyi/Vakhnenko/short pulse equation  is the corresponding equation with cubic nonlinearity
\begin{equation}
	\label{30}
	(u_t - u_{xxx}-(u^3)_x)_x=u.
\end{equation}
These models have attracted a lot of attention in the last thirty years,  as models of water waves under the action of a Coriolis force, \cite{O78, OS90, GOS},  as well as  the amplitude of a ``short'' pulse in an optical fiber, \cite{SW}.
In this paper, we shall be interested in the dynamics of a family of problems, which contains these two. More specifically, we consider the following generalized Ostrovskyi models
\begin{eqnarray}
	\label{ostrovskis}
	\left( u_t -  u_{xxx} - (|u|^p)_x\right)_x = u\\
	\label{10}
	\left( u_t -  u_{xxx} - (|u|^{p-1}u)_x\right)_x = u
\end{eqnarray}
Clearly, \eqref{ostrovskis}, in the case $p=2$ is nothing but \eqref{20}, while \eqref{10}, for $p=3$ is \eqref{30}. 
Let us comment on  the seemingly more general form of the equations that appear in other publications, 
\begin{equation}
\label{r:10}
 \left( u_t -  \be u_{xxx} - \si (|u|^p)_x\right)_x = \ga u, \  \ (x,t) \in  \rone \times \rone 
\end{equation}
Using the scaling transformations $t\to a t, x\to b x, u\to c u$, we obtain the equivalent problem 
$$
 \left( u_t -  \f{\be  b^3}{a}  u_{xxx} - \f{\si  b |c|^{p-2}c }{a} (|u|^p)_x\right)_x =  \f{\ga}{a b} u
$$
which means that by choosing $a,b,c$ appropriately, we may scale all the coefficients to plus or minus 
one, just as in \eqref{ostrovskis}. In addition, by a judicious choice of  the signs of $a,b,c$, one concludes that all systems 
in the form \eqref{r:10} reduce to   
$$
(u_t - \eps_1 sgn(\be)  u_{xxx} - \eps_2   (|u|^p)_x)_x= sgn(\ga) \eps_1 u
$$
 where  $\eps_1, \eps_2  \in \{-1,1\}$.   In this  work, we stick to  the case\footnote{It is well-known that solitary waves do not exists in the case when $sgn(\be)\neq sgn(\ga)$}   $sgn(\be)= sgn(\ga)$. In this case, an appropriate further rescale leads us to 
\begin{equation}
\label{r:20} 
(u_t -   u_{xxx} -  \eps (|u|^p)_x)_x=  u.
\end{equation}
 Thus, our model, \eqref{ostrovskis},  covers the cases for which $\eps=1$.  
 
Let us record another, mostly equivalent formulation of \eqref{ostrovskis} and \eqref{10}. Using $ u = v_x $ in \eqref{ostrovskis} and integrating once (by tacitly assuming that $v, v_x$ vanishes at $\pm \infty$), we  get
\begin{align}\label{o_vx}
	\begin{split}
		(v_t -  v_{xxx} - (|v_x|^p))_x =  v,\\
		(v_t -  v_{xxx} - (|v_x|^{p-1}v_x))_x =  v.
	\end{split}
\end{align}
Regarding local and global  well-posedness for these models, most of the theory has been developed for standard  quadratic and cubic models \eqref{20}, \cite{GL, SW, Ts, VL, LM, Psak}.  Extensive further references to earlier works can be found  in \cite{Ts, VL}. Break up in finite time is shown in various situations in \cite{LPS}. 

The main purpose of this paper is the study of traveling wave solutions, namely functions  in the form $ \phi(x - \om t) $. More specifically, plugging in this ansatz in \eqref{o_vx}   turns it into the profile equation
\begin{align}
	\label{o_profile}
	\begin{split}
		\phi''''+\om \phi''  + \phi + (|\phi'|^p)' = 0, \\
		\phi'''' +  \om \phi''+   \phi+(|\phi'|^{p-1}\phi')' = 0.
	\end{split}
\end{align}
These are fourth order nonlinear ODE's, for which there  is not very well-developed theory. In particular, for non-integer values of $p$, existence has been proved by variational methods, \cite{LL6, LL7, L12}, so that  \eqref{o_profile} is an Euler-Lagrange equation for these constrained minimizers.  Regarding uniqueness, which is well-known to be a hard issue\footnote{even for second order problems of this type}  is only known in the case $p=2$.
This is the main result of  \cite{ZL}, where  it is shown that  localized solutions are unique, together with some asymptotic decay properties of $\phi$ and its derivatives.  Note that the result obtained there rely heavily on the quadratic nonlinearity as well as the precise structure of the equation.   We provide an independent analysis of the elliptic profile equations, \eqref{o_profile} and we also compute, what we believe,  are the sharp spatial exponential rate of decay, see Proposition \ref{prop:exp} below.

Our approach to \eqref{o_profile} is variational, but rather different than the works \cite{L12, LL6, LL7}. More precisely, Levandosky and Liu construct their waves as minimizers of energy, subject to a fixed $L^{p+1}$ norm.
This method  allows for a construction of waves for   any  power of $p >1$.  As shown therein, some of these waves, for large enough $p$,  are spectrally unstable. On the other hand, our goal is to construct the so-called normalized waves - that is, we construct the waves to minimize energy, by keeping their $L^2$ norm fixed.   As we show later in the paper, see Theorem \ref{theo:5}, this imposes restrictions on $p$, but the result is that all of these waves are necessarily spectrally stable. We state our results below, starting with the
existence, and then proceeding onto the stability.
\subsection{Existence of the normalized waves}
Before we present our set up for the existence results, we introduce some basic notations. We use the standard definition of norms in $L^p$  spaces. The Fourier transform and its inverse will be in the form
$$
\hat{u}(\xi)=\f{1}{\sqrt{2\pi}} \int_{\rone} u(x) e^{- i x\xi} dx, \ \ u(x)=\f{1}{\sqrt{2\pi}} \int_{\rone} \hat{u}(\xi) e^{ i x\xi} d\xi.
$$
Consequently, we define all Sobolev norms via $\|f\|_{H^k}=\sum_{j=0}^k \|\p_x^j f\|_{L^2}$ 
for integers $k$, while for values $\al\in {\mathbf R}$ (where $\al$ is not necessarily positive either), we may take the formula
$$
\|f\|_{H^\al}=\left(\int_{\rone} (1+|\xi|^2)^\al |\hat{f}(\xi)|^2d\xi\right)^{1/2}.
$$
Note that for integer $\al$ the two formulations are equivalent. We shall also need the  homogeneous versions of $\dot{H}^\al$, which are defined via the semi-norms
$$
\|f\|_{\dot{H}^\al}=\left(\int_{\rone} |\xi|^{2\al} |\hat{f}(\xi)|^2d\xi\right)^{1/2}. 
$$
Alternatively, a space that will play a role in the arguments below is $\dot{H}^{-1}$, which may be defined as a subspace  of distributions via 
$$
\dot{H}^{-1}=\{f\in \cs'(\rone): f=u_x, \|f\|_{\dot{H}^{-1}}=\|u\|_{L^2}\}.
$$
Similarly, $\dot{H}^{-2}=\{f=u_{xx}: \|f\|_{H^{-2}}=\|u\|_{L^2}\}$ and so on. 
Let us now   introduce the  concrete functionals that we work with, namely
$$
	I[u] = \frac{1}{2} \int_\rone  |u''|^2 +|u|^2dx -\frac{1}{p+1}\int_\rone |u_x|^p u_x dx; \
	J[v] = \frac{1}{2} \int_\rone  |v_{xx}|^2 +|v|^2dx -\frac{1}{p+1}\int_\rone |v_x|^{p+1} dx,
$$
and their variants
$$
	\ci[u] = \frac{1}{2} \int_\rone  |u'|^2 +|\p_x^{-1} u|^2dx -\frac{1}{p+1}\int_\rone |u|^p u dx; \
	\cj[v] = \frac{1}{2} \int_\rone  |v_{x}|^2 +|\p_x^{-1} v|^2 dx -\frac{1}{p+1} \int_\rone |v|^{p+1} dx,
$$
Note $I[u]=\ci[u']$ and $J[v]=\cj[v']$.  For every $\la>0$, we consider the variational problems
\begin{equation}
	\label{40}
	\left\{
	\begin{split}
		I [u]=\frac{1}{2} \int_\rone  |u''|^2 +|u|^2dx -\frac{1}{p+1}\int_\rone |u_x|^p u_x dx \to \min \\
		\int_{\rone} |u'(x)|^2 dx=\la
	\end{split}
	\right.
\end{equation}
\begin{equation}
	\label{41}
	\left\{
	\begin{split}
		\ci [u]=\frac{1}{2} \int_\rone  |u'|^2 +|\p_x^{-1} u|^2dx -\frac{1}{p+1}\int_\rone |u|^p u dx   \to \min \\
		u\in \dot{H}^{-1}, \ \   \int_{\rone} |u(x)|^2 dx=\la
	\end{split}
	\right.
\end{equation}
and
\begin{equation}
	\label{50}
	\left\{
	\begin{split}
		J[v] = \frac{1}{2} \int_\rone  |v''|^2 +|v|^2dx -\frac{1}{p+1}\int_\rone |v_x|^{p+1} dx \to \min \\
		\int_{\rone} |v'(x)|^2 dx=\la
	\end{split}
	\right.
\end{equation}
\begin{equation}
	\label{51}
	\left\{
	\begin{split}
		\cj[v] =  \frac{1}{2} \int_\rone  |v_{x}|^2 +|\p_x^{-1} v|^2 dx -\frac{1}{p+1} \int_\rone |v|^{p+1} dx\to \min \\
		v\in \dot{H}^{-1},\ \  \int_{\rone} |v(x)|^2 dx=\la
	\end{split}
	\right.
\end{equation}
Our existence results are as follows.
\begin{theorem}[Existence of solitary waves]
	\label{theo:5}
	Let $\la>0$. Then,
	\begin{itemize}
		\item For $1<p<3$,   the constrained minimization problems \eqref{40} and \eqref{41}  have  solutions $\vp_\la$ and $\phi_\la$.  In addition, $\phi_\la\in H^2\cap \dot{H}^{-2}(\rone), \vp_\la\in H^4(\rone):   \vp_\la'=\phi_\la$  and they satisfy, for some $\om=\om_\la \in (-\infty, 2)$, 
		      \begin{eqnarray*}
			      & & 	\p_x^2 \phi_\la+\p_x^{-2} \phi_\la+\om \phi_\la+|\phi_\la|^{p}=0, \\
			      & & 	\vp_\la''''+\om \vp_\la''+\vp_\la+(|\vp_\la'|^p)'=0.
		      \end{eqnarray*}
		      respectively.   The waves $\vp_\la, \phi_\la$ are exponentially decaying, together with their derivatives, in fact
		      \begin{equation}
			      \label{dec}
			      |\vp_\la(x)|+|\phi_\la(x)|+|\phi'_\la(x)|\leq C e^{-\ka_{\om} |x|}, k_{\om}:=\left\{\begin{array}{cc}
				      \f{\sqrt{2-\om}}{2}                   & \om \in (-2,2), \\
				      \sqrt{\f{-\om-\sqrt{\om^2-4}}{2}} & \om <-2
			      \end{array}
			      \right.
		      \end{equation}

%		      Furthermore, for every $\de_k\to 0$, there is a subsequence $\de_{k_j}$ and
%		      $U\in H^1\cap \dot{H}^{-1}$, so that    $\lim_j \|\phi_{\la+\de_{k_j}}- U\|_{H^1\cap \dot{H}^{-1}}=0$. Furthermore,  $U$ is a constrained minimizer for \eqref{41}.  Similar statements for \eqref{40}, except $\lim_j \|\vp_{\la+\de_{k_j}}-U\|_{H^2}=0$.
		\item For $1<p<5$, the minimization problems \eqref{50} and \eqref{51} and  have constrained minimizers, $\vp_\la\in H^4(\rone), \phi_\la\in H^2\cap \dot{H}^{-2}(\rone):\vp_\la'=\phi_\la$,   which satisfy the
		      \begin{eqnarray*}
			      & & 	\p_x^2 \phi_\la+\p_x^{-2} \phi_\la+\om \phi_\la+|\phi_\la|^{p-1} \phi_\la=0, \\
			      & & 	\vp_\la''''+\om \vp_\la''+\vp_\la+(|\vp_\la'|^{p-1} \vp'_\la)'=0.
		      \end{eqnarray*}
		      The waves $\vp_\la, \phi_\la$ satisfy \eqref{dec} and similar statements hold for the waves $U$. 
	\end{itemize}

\end{theorem} 
\begin{itemize}
	\item The waves  $\vp_\la, \phi_\la$, which  are constructed to  satisfy the Euler-Lagrange equation  in a weak sense, are actually smoother solutions, see Proposition  \ref{prop:nm} below.
	\item The waves satisfy the decay bounds \eqref{dec} hold whenever one has a weak solution of \eqref{o_profile}, see Proposition \ref{prop:exp}. This result matches the results in Zhang-Liu, \cite{ZL}, see Lemma 3.2 and Remark 3.1, p. 824,   for the case $\om<-2$.  For the case $\om \in (-2,2)$, the  new bound \eqref{dec} provides the sharp rate of decay for the solitary waves.   
	\item Our method does not provide a definite answer to the following important question - does there exist a wave for each value of $\om \in (-\infty,2)$? Recall that this is due to the fact that we give ourselves a $\la$, which in turn produces $\om_\la\in (-\infty,2)$. It remains unclear though whether any arbitrary $\om\in (-\infty,2)$ corresponds to such $\la>0$.  
	\item If one knows {\it a priori} that the Euler-Lagrange equation has unique solution, for a fixed value of $\om=\om_\la$, then this solution is of course exactly the limit wave at that $\la$.  As we have mentioned above, the uniqueness is only known   for the case $p=2$ in \cite{ZL}.
\end{itemize}
Let us again point out that in \cite{LL6, LL7, L12}, the authors have constructed traveling waves for values of $p$ beyond the range of Theorem \ref{theo:5}, due to the  use an alternative variational approach.

\subsection{Stability results}
We start by describing in detail the state of the art, regarding the stability of the Ostrovsky  waves. For the reduced Ostrovsky case, that is the model without the dispersive term $u_{xxx}$ and with quadratic or cubic non-linearities, much is known, as the model is in some sense completely integrable. A full description of its  periodic waves as well their stability can be found in the recent papers, \cite{GP, JP}.  

For the full Ostrovsky model under consideration,   Liu and Ohta, \cite{LO8} and by a slightly different method, Liu, \cite{LO7} have established the orbital stability for the classical Ostrovsky's equation (i.e. $p=2$) for large  speeds.    Another,  set stability  result, sometimes referred to as weak orbital stability,   is given in \cite{LV04}.
In the works, \cite{LL6}, \cite{LL7}, Levandosky and Liu have constructed the waves for the generalized problems and they have shown that their orbital stability is reduced to the convexity of certain scalar functions, a la Grillakis-Shatah-Strauss. In  \cite{L12}, Levandosky obtained rigorously the orbital stability of the waves near some bifurcation points. In addition, he has launched an impressive numerical study, which was our main motivation for this work.

In order to state our stability results, we need to introduce the linearized operators as well. Namely, for a traveling wave $\phi$, solving either one of the elliptic equations in \eqref{o_profile}, set
$ u(t,x) = \phi(x-\om t) + v(t,x-\om t) $ into  \eqref{o_vx}. After  ignoring  $ O(v^2) $ terms we get
\begin{equation}\label{perturb}
	(v_x)_t  -v_{xxxx}-\om v_{xx} -v - p(|\phi'|^{p-2}\phi' v_x)_x = 0.
\end{equation}
Setting the stability ansatz $ v(t,x) = e^{t \mu} z(x) $ in  \eqref{perturb},  we obtain the eigenvalue problem  in the form
\begin{equation}
	\label{318}
	L_+ z =\mu  \p_x z, \ \  L_+=\p_{xxxx}+\om \p_{xx}+1 + p \p_x(|\phi'|^{p-2}\phi'\p_x(\cdot)).
\end{equation}
Clearly, $(L_+, D(L_+))=H^4(\rone)$ is unbounded, but self-adjoint operator on $L^2(\rone)$. Spectral instability here is understood as the existence of a non-trivial pair $(\mu, z): \Re \mu>0, z\neq 0, z\in D(L_+)$, so that \eqref{318} is satisfied. Spectral stability means non-existence of such pair.

The eigenvalue problem \eqref{318} is a non-standard one, although problems in this form were  recently considered  in the literature. An equivalent formulation, which is technically more convenient for our approach is the following: write $L_+=-\p_x \cl_+ \p_x$, where $ \cl_+ = -\p_x^{2} - \om- \p_x^{-2}   - p |\phi'|^{p-2}\phi' $,   $D(\cl_+)=H^2(\rone)\cap \dot{H}^{-2}(\rone)$. 

In terms of the new operator $\mu z_x=-\p_x \cl_+ z_x$. Since the function spaces require vanishing at both infinities, this is equivalent to $\mu z=- \cl_+ \p_x z$ or $\mu$ is an eigenvalue for $-\cl_+ \p_x$. Equivalently, $\bar{\mu}$ is an eigenvalue for the adjoint $\p_x \cl_+=(-\cl_+\p_x)^*$ or $\p_x \cl_+ z=\bar{\mu} z$. Clearly, since $\Re\mu=\Re\bar{\mu}$, the spectral stability of the traveling wave $\phi(x-\om t)$ is equivalent to  the following eigenvalue problem 
\begin{equation}
\label{319}
\p_x \cl_+ z=\la z
\end{equation}
not having non-trivial solutions $(\la, z), \Re \la>0, z\neq 0$. 

 For future reference, we introduce two notions of non-degeneracy of $\phi$, which are generically expected to hold. 
\begin{definition}
	\label{defi:ol} We say that the wave $\phi$ is non-degenerate, if $Ker[\cl_+]=span[\phi']$. 
We say that the wave $\phi$ is weakly non-degenerate, if $\phi\perp Ker[\cl_+]$. 
\end{definition}
\noindent {\bf Remark:} Clearly a non-degenerate wave is weakly non-degenerate, but not vice versa. For weakly non-degenerate waves, the important quantity $\cl_+^{-1} \phi$ is well-defined. 
\begin{theorem}
	\label{theo:10}
	Let  $1< p<3$ and  $\la>0$.  Then, the constrained minimizers  $\phi_\la$ for \eqref{40}, constructed in Theorem \ref{theo:5}  are weakly non-degenerate in the sense of Definition \ref{defi:ol}. In particular, $\cl_+^{-1} \phi_\la$ is well-defined. Under the condition 
		\begin{equation}
		\label{nond} 
		\dpr{\cl_+^{-1} \phi}{\phi}\neq 0, 
		\end{equation}
the waves $\phi_\la$ are   spectrally stable.  That is, the eigenvalue problem 
 \eqref{319} does not support  non-trivial solutions $(\la, z): \Re \la>0, z\neq 0$. 
 Similar result holds for limit waves of \eqref{40}.

	For $1< p<5$ and $\la>0$,  let $\phi_\la$ be a constrained minimizer for  \eqref{51}. 
	Then, again $\phi_\la$ are weakly non-degenerate and  under \eqref{nond}, they are spectrally stable as well.  Same statement applies to the limit waves. 
\end{theorem}
{\bf Remark:} The condition \eqref{nond} is a commonly imposed technical  condition in the literature. In fact, it guarantees that the zero eigenvalue for $\p_x \cl_+$ has an algebraic multiplicity of exactly two, see Proposition \ref{prop:lk} below.

 The paper is organized as follows. In Section \ref{sec:2}, we state some preliminary results and background information - most  of these are well-known or standard, yet others, like Proposition \ref{prop:exp} require quite a bit of analysis.  Section \ref{sec:3} contains the main variational construction, which proceeds through  compensated compactness arguments. The argument here requires a surprising detailed, hands on Fourier analytic methods in order to verify the subadditivity of the cost functional.  In addition, Euler-Lagrange equations are derived and some basic spectral properties of the linearized operators are verified as well. In Section \ref{sec:4}, we present the basics of the instability index count theory, along with some corollaries, which fit our situation well. In addition, we establish a general framework, which yields the spectral stability of the waves, Lemma \ref{stab:10}. This is then easily applied to the previously waves to infer their spectral stability. 
\section{Preliminaries}
\label{sec:2}

\subsection{Weak solutions and bootstrapping regularity}
In our considerations, we will need to  rely, at least initially,  on a weak solution formulations of certain elliptic PDE's, specifically  \eqref{o_profile}. More concretely, 
\begin{definition}
	\label{defi:10}
	We say that $g\in H^2(\rone)$ is a weak solution of the equation
	\begin{equation}
		\label{515}
		g''''+\om  g''+g+(F(g'))'=0,
	\end{equation}
	if the non-linearity satisfies $F(f')\in L^2$, whenever $f\in H^2$ and for every $h\in H^2$, we have the relation
	$\dpr{g''}{h''}+  \dpr{\om g''+g}{h} - \dpr{F(g')}{h'}=0$.
\end{definition}
A simple observation is that if $g$ is a weak solution of \eqref{515}, in the sense of Definition \ref{defi:10}, then we can bootstrap its smoothness, namely $g\in H^3(\rone)$. Indeed, since the operator $\p_x^4+1$ is invertible on $L^2(\rone)$, introduce $\tilde{g}:=(\p_x^4+1)^{-1}[-\om g''+\p_x(F(g'))]\in L^2(\rone)$. Of course, this is the {\it formal}  solution  of \eqref{515}, which should mean that $\tilde{g}=g$, which we will prove momentarily. Before that, let us observe that due to the smoothing nature of $(\p_x^4+1)^{-1}:L^2\to H^4$, we can immediately see that\footnote{which can be improved further to $\tilde{g} \in H^4$, once we impose the  mild extra smoothness assumption  $F(g)\in H^1$. This  will not be necessary for our purposes.}  $\tilde{g} \in H^3(\rone)$. Now, for every test function $h$, we have that\footnote{understood as pairing between an element of the distribution space $H^{-2}$ and $H^2$}
$$
	\dpr{(1+\p_x^4) g}{h}=-\om \dpr{g''}{h}-\dpr{F(g')}{h'}=\dpr{(1+\p_x^4) \tilde{g}}{h}
$$
It follows that $\dpr{ g}{(1+\p_x^4)h}=\dpr{\tilde{g}}{(1+\p_x^4)h}$ for all $h$, whence $g=\tilde{g}$. In particular, we have shown the extra regularity $g\in H^3(\rone)$. One can immediately bootstrap this to $g\in H^4(\rone)$ by taking into account the representation  $g =(\p_x^4+1)^{-1}[-\om g''+\p_x(F(g'))]$, if $\p_x F(g')\in L^2$. This is the case for the   profile equations \eqref{o_profile}.   Thus, we have shown
\begin{proposition}
	\label{prop:nm}
	The weak solution $g$ of \eqref{515} is in fact $g\in H^3(\rone)$. For non-linearities in the form
	$F(z)=|z|^p, |z|^{p-1} z$, this can be further improved to $g \in H^4(\rone)$, whence the weak solutions of \eqref{515}   in fact satisfy \eqref{515} as $L^2$ functions.
\end{proposition}
Due to this result, we will henceforth not make the distinction between weak and strong(er) solutions of our profile equations.
\subsection{Exponential decay of the waves and  eigenfunctions}
In this section, we show that the solutions to the elliptic profile equations \eqref{o_profile} have exponential decay at $\pm \infty$, and in fact we are able to compute explicitly the leading order terms. Similar result holds for any element in the kernels of  the linearized operators $\cl_+, L_+$. The precise result is as follows.
\begin{proposition}
	\label{prop:exp}
	Let $\phi\in H^4$ solves either of the  fourth order    profile equations  \eqref{o_profile}, with $\om<2$. Then, $\phi, \phi'$ both have  exponential decay at $\pm \infty$ and in fact,
	\begin{equation}
		\label{d:exp}
		|\phi(x)|+|\phi'(x)|\leq C e^{-k_\om|x|}, k_{\om}:=\left\{\begin{array}{cc}
			\f{\sqrt{2-\om}}{2}                   & \om \in (-2,2), \\
			\sqrt{\f{-\om-\sqrt{\om^2-4}}{2}} & \om <-2
		\end{array}
		\right.
	\end{equation}

	In addition, every eigenfunction of $L_+ \Psi=0$ has the same exponential decay.
	Similarly, let $\phi\in H^2\cap \dot{H}^{-2}$ solves, for $\om<2$,
	$$
		\p_x^2 \phi +\p_x^{-2} \phi +\om \phi+|\phi_\la|^{p}=0,
	$$
	Then, $\phi$ has the same exponential decay as in \eqref{d:exp}, together with the eigenfunctions corresponding to zero eigenvalues for $\cl_+$.
\end{proposition}
 \begin{proof} 
 We work with the fourth order waves, namely the solutions of \eqref{o_profile}. Noting that $\xi^4-\om \xi^2+1>0$, for every $\xi\in\rone$, since $\om<2$, we have that $(\p_x^4+\om \p_x^2+1)^{-1}$ is a bounded operator on $L^2$, so we have the  representation
 \begin{eqnarray*}
 	\phi &=& -(\p_x^4+\om \p_x^2+1)^{-1}[|\phi'|^{p-1} \phi'], \\
 	\phi &=& -(\p_x^4+\om \p_x^2+1)^{-1}[|\phi'|^{p}],
 \end{eqnarray*}
 Take a derivative in this last equation and denote $g:=\phi'$, so
 \begin{eqnarray}
 \label{app:10}
 g &=& -\p_x (\p_x^4+\om \p_x^2+1)^{-1}[|g|^{p-1} g] \\
 \label{app:11}
 g &=& -\p_x (\p_x^4+\om \p_x^2+1)^{-1}[|g|^{p}]
 \end{eqnarray}
 Clearly, it is enough to show the desired exponential decay for $g$, whence since $\phi$ vanishes at $\pm \infty$, one can conclude from the representations
 $\phi(x)=-\int_x^\infty \phi'(y) dy=\int_{-\infty}^x \phi'(y) dy$, that $\phi$ vanishes at the same rate at $\pm \infty$.
 Let  $V(x):=|g(x)|^{p-1}$ or   $V(x):=|g(x)|^{p-1} sgn(g(x))$, depending on whether we consider \eqref{app:10} or \eqref{app:11}. Either way, we consider
 \begin{equation}
 \label{app:15}
 g= -\p_x (\p_x^4+\om \p_x^2+1)^{-1}[V g]
 \end{equation}
 where $\lim_{|x|\to \infty} V(x)=0$, since $g\in H^4\subset C_0(\rone)$.
 
 Let us now comment on the operator
 $(\p_x^4+\om \p_x^2+1)^{-1}$. Clearly
 $$
 (\p_x^4+\om \p_x^2+1)^{-1} f(x)=\int_{-\infty}^\infty G_\om(x-y) f(y) dy,
 $$
 where $\widehat{G}_\om(\xi)=\f{1}{\xi^4-\om \xi^2 +1}$. Note that since the roots of the bi-quadratic equation $\ka^4-\om_\la \ka^2+1=0$ are given by
 \begin{equation}
 \label{k:10}
 k^2= \f{\om \pm \sqrt{\om^2-4}}{2}.
 \end{equation}
 By the formula  $\widehat{(\xi^2+k^2)^{-1}}=\f{\pi}{k} e^{-k |\cdot|}$, we see that $G$ is a linear combination of two exponential functions.  In fact, taking into account $\om<2$ and after some elementary analysis,
 we conclude that the solutions of \eqref{k:10}, have $\Re k=\f{\sqrt{2-\om}}{2}$ when $\om \in (-2,2)$ and
 $ \Re k=\sqrt{\f{-\om-\sqrt{\om^2-4}}{2}}$ for $\om<-2$. It follows that
 \begin{eqnarray*}
 	& & |G_\om(x)|+|G'_\om(x)|\leq C_\om  e^{- k_\om |x|}, k_\om:=\left\{\begin{array}{cc}
 		\f{\sqrt{2-\om}}{2}               & \om \in (-2,2), \\
 		\sqrt{\f{-\om-\sqrt{\om^2-4}}{2}} & \om<-2
 	\end{array}
 	\right.
 \end{eqnarray*}
 We are now ready to analyze \eqref{app:15}. To this end,  let $\eps=\eps_\om>0$, to be selected momentarily. Let $N$ be so large that $|V(x)|<\eps$, so long as $|x|>N$. We now rewrite \eqref{app:15} as
 \begin{equation}
 \label{app:20}
 g+\int_{|y|>N} G'_\om(x-y) V(y) g(y) dy  = -\int_{|y|\leq N} G'_\om(x-y) V(y) g(y) dy
 \end{equation}
 We can view this as an integral equation in $X_N:=L^\infty(|\cdot|>N)$, with \\ $\cg g(x) =\chi_{|x|>N} \int_{|y|>N} G'_\om(x-y) V(y) g(y) dy$ acting boundedly on $X_N$. In fact, for every $m: 0\leq m\leq k_\om$ and every $g\in Y_m: \|g\|_{Y_m}:=\sup_{|x|>N} |g(x)| e^{m |x|}<\infty$, we have\footnote{Here, we use the fact that for $0<a<b$, $\int_{-\infty}^\infty e^{-a|y|} e^{-b|x-y|} dy\leq
 	C_{b} e^{-a|x|}$}
 \begin{eqnarray*}
 	|\cg g (x)|  &\leq &  C_\om \eps \|g\|_{Y_m} \int_{|y|>N} e^{-k_\om|x-y|} e^{-m|y|} dy \leq C_\om \eps \|g\|_{Y_m}  \int_{-\infty}^\infty
 	e^{-k_\om|x-y|} e^{-m|y|} dy\leq \\
 	&\leq & \eps D_\om  \|g\|_{Y_m} e^{-m |x|},
 \end{eqnarray*}
 whence $\cg:Y_m\to Y_m$ with $\|\cg\|_{B(Y_m)}\leq \eps D_\om$. Thus, select $\eps(\om): \eps D_\om=\f{1}{2}$.
 
 In the particular case $m=0$, we can use von Neumann series to resolve \eqref{app:20}
 \begin{equation}
 \label{app:40}
 g=\sum_{k=0}^\infty (-1)^{k+1}  \cg^k[\int_{|y|\leq N} G'_\om(\cdot-y) V(y) g(y) dy].
 \end{equation}
 Using the representation \eqref{app:40}, the fact that $|\int_{|y|\leq N} G'_\om(x-y) V(y) g(y) dy|\leq C_\om  e^{-k_\om|x|}$,  and by the mapping properties of $\cg$, we conclude that $g\in Y_{k_\om}$. That is,
 $
 \sup_{|x|>N} |g(x)|\leq C e^{-k_\om|x|},
 $
 which by the boundedness of $g$ can be extended to $\sup_x |g(x)|\leq C e^{-k_\om|x|}$.
 
 Regarding the eigenfunctions, we employ the same strategy, namely if $L_+ \Psi=0$, this means that for $g=\Psi'$,
 \begin{equation}
 \label{eq:g}
 g(x)= - p \int_{-\infty}^\infty G_\om'(x-y)[V(y) g(y)] dy,
 \end{equation}
 with the same $V$ as above. Due to the fact that $|V|\sim |\phi'(x)|^{p-1}$ has exponential decay now, clearly \eqref{eq:g} can be bootstrapped to produce decay for $g$ matching the decay of $G_\om'$, that is $e^{-k_\om |x|}$. Finally, $\Psi(x)= - \int_x^\infty g(y) dy=\int_{-\infty}^x g(y) dy$ recovers the same exponential decay for $\Psi$ as for $g$.
\end{proof}
\subsection{Pohozaev identities}
\begin{lemma}
	\label{le:poh1}
	Suppose $\phi \in H^2(\rone)$ is a weak solution of
	\begin{equation}
		\phi''''+\om \phi''+\phi+\partial_{x}(|\phi'|^{p-1}\phi')=0.
		\label{110}
	\end{equation}
	More concretely,  for every test function $h\in H^2(\rone)$, $\dpr{\phi''}{h''}+
		\om	\dpr{ \phi''}{h}-\dpr{|\phi'|^{p-1}\phi'}{\partial_{x}h}=0$.
	Then,  the following identities hold
	\begin{equation}
		\begin{split}\int_{\rone}\left|\phi''\right|^{2}dx & =\int_{\rone}\left|\phi\right|^{2}dx+\frac{p-1}{2(p+1)}\int\left|\phi'\right|^{p+1}dx,\\
			\omega \int\left|\phi'\right|^{2}dx & =2\int_{\rone}\left|\phi\right|^{2}dx-\frac{p+3}{2(p+1)}\int\left|\phi'\right|^{p+1}dx.
		\end{split}
		\label{120}
	\end{equation}
	Similarly, suppose  $\phi\in H^2(\rone)$ is a weak solution of
	\begin{equation}
		\phi''''+\om \phi''+\phi+\partial_{x}(|\phi'|^{p})=0,\label{130}
	\end{equation}
	then
	\begin{equation}
		\begin{split}\int_{\rone}\left|\phi''\right|^{2}dx & =\int_{\rone}\left|\phi\right|^{2}dx+\frac{p-1}{2(p+1)}\int\left|\phi'\right|^{p}\phi'dx,\\
			\omega \int\left|\phi'\right|^{2}dx & =2\int_{\rone}\left|\phi\right|^{2}dx-\frac{p+3}{2(p+1)}\int\left|\phi'\right|^{p}\phi'dx.
		\end{split}
		\label{140}
	\end{equation}
\end{lemma}
\begin{proof}
	Multiplying \eqref{110}  by $\phi$ and integrating over
	$\rone$ we get
	\begin{equation}
		\int_{\rone}\left|\phi''\right|^{2}dx-\omega \int\left|\phi'\right|^{2}dx+
		\int_{\rone}\left|\phi\right|^{2}dx-\int\left|\phi'\right|^{p+1}dx=0.\label{150}
	\end{equation}
	Now, multiplying \eqref{110}  by $x\phi'$ (recall that according to Proposition \ref{prop:exp} this function has exponential decay) and integrating
	over $\rone$ we get
	\begin{equation}
		\frac{3}{2}\int_{\rone}\left|\phi''\right|^{2}dx-
		\frac{\omega}{2}\int_{\rone}\left|\phi'\right|^{2}dx-\frac{1}{2}\int_{\rone}\left|\phi\right|^{2}dx-\frac{p}{p+1}\int_{\rone}\left|\phi'\right|^{p+1}dx=0.
		\label{160}
	\end{equation}
	Solving \eqref{150}  and \eqref{160}  for $\int_{\rone}\left|\phi''\right|^{2}dx$
	and $\omega \int\left|\phi'\right|^{2}dx$ we get \eqref{120}. Finally, the proof of \eqref{140} follows similar path.
	%	 For the proof of \eqref{160}, we have two issues - lack of enough regularity, since {\it a priori}  $\phi'\in H^1$, and also lack of enough decay as $x\phi'$ is not necessarily decaying at $\pm \infty$. For this, we identify $\phi=(1+\p_x^4)^{-1}[-\om \phi''-\partial_{x}(|\phi'|^{p-1}\phi')]$, whence by the smoothing properties of $(1+\p_x^4)^{-1}$,  it is easy to see that $\phi$ is in fact smoother, say  $\phi\in H^3$. Thus, we can run an alternative proof of \eqref{160} by taking $h:=x\phi' \chi(x/N)\in H^2$, where $\chi\in C_0^\infty, \chi(z)=1, |z|<1$ and $N>>1$. After performing the integration by parts, and taking appropriate limits in $N\to \infty$, we see that \eqref{160} still holds as stated. 

\end{proof}
\noindent An easy corollary of Lemma \ref{le:poh1} is the following lemma.
\begin{lemma}
	\label{le:poh2}
	Suppose $\phi \in H^1(\rone)\cap \dot{H}^{-1}(\rone)$ is a weak solution of
	\begin{equation}
		\phi''+\p_x^{-2} \phi+\om \phi+ |\phi|^{p-1}\phi=0.
		\label{1101}
	\end{equation}
	Then,  the following identities hold
	\begin{equation}
		\begin{split}\int_{\rone}\left|\phi'\right|^{2}dx & =\int_{\rone}\left|\p_x^{-1} \phi\right|^{2}dx+\frac{p-1}{2(p+1)}\int\left|\phi\right|^{p+1}dx,\\
			\omega \int\left|\phi\right|^{2}dx & =2\int_{\rone}\left|\p_x^{-1} \phi\right|^{2}dx-\frac{p+3}{2(p+1)}\int\left|\phi\right|^{p+1}dx.
		\end{split}
		\label{1201}
	\end{equation}
	Similarly, suppose  $\phi\in H^1(\rone)\cap \dot{H}^{-1}(\rone)$ is a weak solution of
	\begin{equation}
		\label{1301}
		\phi''+\p_x^{-2} \phi+\om \phi+\partial_{x}(|\phi|^{p})=0,
	\end{equation}
	then
	\begin{equation}
		\begin{split}\int_{\rone}\left|\phi'\right|^{2}dx & =\int_{\rone}\left|\p_x^{-1} \phi\right|^{2}dx+\frac{p-1}{2(p+1)}\int\left|\phi\right|^{p}\phi dx,\\
			\omega \int\left|\phi\right|^{2}dx & =2\int_{\rone}\left|\p_x^{-1} \phi\right|^{2}dx-\frac{p+3}{2(p+1)}\int\left|\phi\right|^{p}\phi dx.
		\end{split}
		\label{1401}
	\end{equation}
\end{lemma}
\begin{proof}
	Just apply Lemma \ref{le:poh1}  to the function $g$, where $\phi=g'$. Note that $g\in H^2$ solves \eqref{110} or \eqref{130}.
\end{proof}

\subsection{Sampling a $W^{1,1}$ function}
We have the following elementary  lemma, which may be of independent interest.
\begin{lemma}\label{portion}
	Let $ N> 1 $ be an integer and $ f:\rone\to \rone $,  $f\in W^{1,1}(\rone)$. Then
	\begin{equation*}%\label{}
		\sum_{n=-\infty }^{\infty}\int_{n\varepsilon}^{n\varepsilon + \frac{\varepsilon}{N}}f(x)dx = \frac{1}{N} \int_{\rone} f(x)dx +O(\varepsilon)
	\end{equation*}
	as $ \varepsilon \rightarrow 0+ $. More precisely,
	$$
		\left| \sum_{n=-\infty }^{\infty}\int_{n\varepsilon}^{n\varepsilon + \frac{\varepsilon}{N}}f(x)dx - \frac{1}{N} \int_{\rone} f(x)dx\right|\leq  \f{\eps}{N}  \int_{\rone} |f'(y)| dy
	$$
\end{lemma}

\begin{proof}
	Let's split each interval $ [n\varepsilon + \frac{\varepsilon}{N}, (n+1)\varepsilon) $ into $ N-1 $ equal intervals and compare  one of them with the integral over the interval $ [n\varepsilon,n\varepsilon +\frac{\varepsilon}{N}) $. We have
	\begin{eqnarray*}
		& & 	\left |\int_{n\varepsilon}^{n\varepsilon +\frac{\varepsilon}{N}}f(x)dx - \int_{n\varepsilon+\frac{m\varepsilon}{N}}^{n\varepsilon +\frac{(m+1)\varepsilon}{N}}f(x)dx\right |
		= \left |\int_{n\varepsilon}^{n\varepsilon +\frac{\varepsilon}{N}}f(x)- f\left (x+\frac{m\varepsilon}{N}\right )dx \right |\\
		&\leq & \int_{n\varepsilon}^{n\varepsilon +\frac{\varepsilon}{N}}\int_{x}^{x+\frac{m\varepsilon}{N}}|f'(y)|dydx \leq  \int_{n\varepsilon}^{n\varepsilon +\frac{\varepsilon}{N}}\int_{n\varepsilon}^{n\varepsilon+\frac{(m+1)\varepsilon}{N}}|f'(y)|dydx\\
		&\leq & \frac{\varepsilon}{N} \int_{n\varepsilon}^{(n+1)\varepsilon}|f'(y)|dy
	\end{eqnarray*}
	for all $ m = 1,\ldots,N-1 $.

	Now, using this last estimate, we get, after adding and subtracting $\sum_{m=1}^{N-1}\int_{n\varepsilon+\frac{m\varepsilon}{N}}^{n\varepsilon +\frac{(m+1)\varepsilon}{N}}f(x)dx$,
	\begin{eqnarray*}
		& & 	N\sum_{n=-\infty }^{\infty}\int_{n\varepsilon}^{n\varepsilon + \frac{\varepsilon}{N}}f(x)dx
		= \sum_{n=-\infty }^{\infty} \left(\int_{n\varepsilon}^{n\varepsilon + \frac{\varepsilon}{N}}f(x)dx +  \sum_{m=1}^{N-1}\int_{n\varepsilon+\frac{m\varepsilon}{N}}^{n\varepsilon +\frac{(m+1)\varepsilon}{N}}f(x)dx\right)\\
		& + & \sum_{n=-\infty }^{\infty} \left((N-1)\int_{n\varepsilon}^{n\varepsilon + \frac{\varepsilon}{N}}f(x)dx - \sum_{m=1}^{N-1}\int_{n\varepsilon+\frac{m\varepsilon}{N}}^{n\varepsilon +\frac{(m+1)\varepsilon}{N}}f(x)dx\right)= \\
		&=& \sum_{n=-\infty }^{\infty} \int_{n\varepsilon}^{(n+1)\varepsilon }f(x)dx +
		\sum_{n=-\infty }^{\infty} \sum_{m=1}^{N-1} \left( \int_{n\varepsilon}^{n\varepsilon +\frac{\varepsilon}{N}}f(x)dx - \int_{n\varepsilon+\frac{m\varepsilon}{N}}^{n\varepsilon +\frac{(m+1)\varepsilon}{N}}f(x)dx\right) =\\
		&=& \int_{\rone}  f(x) dx +
		\sum_{n=-\infty }^{\infty} \sum_{m=1}^{N-1} \left( \int_{n\varepsilon}^{n\varepsilon +\frac{\varepsilon}{N}}f(x)dx - \int_{n\varepsilon+\frac{m\varepsilon}{N}}^{n\varepsilon +\frac{(m+1)\varepsilon}{N}}f(x)dx\right).
	\end{eqnarray*}
	Rearranging terms and using the estimate for  $\left |\int_{n\varepsilon}^{n\varepsilon +\frac{\varepsilon}{N}}f(x)dx - \int_{n\varepsilon+\frac{m\varepsilon}{N}}^{n\varepsilon +\frac{(m+1)\varepsilon}{N}}f(x)dx\right |$, implies
	\begin{eqnarray*}
		& & |N\sum_{n=-\infty }^{\infty}\int_{n\varepsilon}^{n\varepsilon + \frac{\varepsilon}{N}}f(x)dx - \int_{\rone}  f(x) dx|\leq \sum_{n=-\infty }^{\infty} \sum_{m=1}^{N-1} \left| \int_{n\varepsilon}^{n\varepsilon +\frac{\varepsilon}{N}}f(x)dx - \int_{n\varepsilon+\frac{m\varepsilon}{N}}^{n\varepsilon +\frac{(m+1)\varepsilon}{N}}f(x)dx\right|\leq \\
		&\leq & \eps \sum_{n=-\infty }^{\infty} \int_{n\varepsilon}^{(n+1)\varepsilon}|f'(y)|dy=\eps \int_{\rone} |f'(y)| dy.
	\end{eqnarray*}
	Dividing by $N$ yields the claim.
\end{proof}

\section{Variational construction}
\label{sec:3}
In this section, we provide the variational construction of the waves.
It turns out that for some aspects of the construction, it is more beneficial to look at the following alternative
$\ci, \cj$ defined in the beginning.  Introduce  the following functions, which are the corresponding infimums, if they exists, of the constrained minimization problems
\begin{align}
	 & m_I(\lambda) = \inf_{u\in H^{2},\norm{u_{x}}^{2}=\la } \left \{\frac{1}{2} \int_\rone  |u_{xx}|^2 +|u|^2dx -\frac{1}{p+1}\int_\rone |u_x|^p u_x dx\right \},\label{minI}                       \\
	 & m_J(\lambda) = \inf_{v\in H^{2},\norm{v_{x}}^{2}=\la } \left \{\frac{1}{2} \int_\rone  |v_{xx}|^2 +|v|^2dx -\frac{1}{p+1}\int_\rone |v_x|^{p+1} dx\right \},\label{minJ}                       \\
	 & m_\ci(\lambda) = \inf_{U\in H^{1}\cap \dot{H}^{-1},\norm{U}^{2}=\la } \left \{\frac{1}{2} \int_\rone  |U_x|^2 +|\p^{-1}_x U|^2dx -\frac{1}{p+1}\int_\rone |U|^p U dx\right \},\label{minIp}    \\
	 & m_\cj(\lambda) = \inf_{V\in H^{1}\cap \dot{H}^{-1},\norm{V}^{2}=\la  } \left \{\frac{1}{2} \int_\rone  |V_x|^2 +|\p^{-1}_x V|^2dx -\frac{1}{p+1}\int_\rone |V|^{p+1} dx\right \}.\label{minJp}
\end{align}
These are usually referred to as cost functions.
We have the following sequence of lemmas, that establishes some important properties of the functionals and the $m$ functions.
\subsection{The variational problems are well-posed and equivalent}
Regarding well-posedness,  we have the following result.
\begin{lemma}
	\label{le:20}
	For $1<p<5$, 	$ m_I, m_J>-\infty $. That is, the problems \eqref{40} and \eqref{50} are well-posed.
\end{lemma}
\begin{proof}
	Indeed, it is simple to see that $m_I\geq m_J$.  From the GNS inequality,
	$$
		\| u_x\|_{L^{p+1}}^{p+1}\leq C \|u_x\|_{\dot{H}^{\f{p-1}{2}}}^{p+1}\leq C \norm{u_x}^{\frac{p+3}{2}}_2 \norm{u_{xx}}^{\frac{p-1}{2}}_2,
	$$
	we have
	\begin{align*}
		I[u]\geq J[u]\geq  \frac{1}{2} \int_\rone |u_{xx}|^2dx - C \norm{u_x}^{\frac{p+3}{2}}_2 \norm{u_{xx}}^{\frac{p-1}{2}}_2,
	\end{align*}
	Clearly, if $p\in (1,5)$, $\f{p-1}{2}<2$, so we can use Young's and absorb
	$\norm{u_{xx}}^{\frac{p-1}{2}}_{L^2}$. Thus, we get a bound
	$$
		I[u]\geq J[u]\geq C_\la.
	$$
\end{proof}
The next result is about the equivalence of $m_I, m_\ci$, and $m_J, m_\cj$ respectively.
\begin{lemma}
	\label{le:3}
	For $ 1<p<5  $ we have that $ m_I(\la)=m_\ci(\la) $ and $ m_J(\la) = m_\cj(\la) $. Moreover, if $ \varphi_\la $ is a minimizer for $ m_I(\la) $ ($ m_J(\la) $ respectively), then $ \phi_\la = \varphi_\la' $ is a minimizer for $ m_\ci(\la) $($ m_\cj(\la) $ respectively).
\end{lemma}
\begin{proof}
	Let $ \phi$ be a Schwartz function, so that $\hat{\phi}$ is compactly supported, with the property that there exists a $ \de >0 $, so that $ \widehat{\phi} (\xi ) = 0 $ for all $ |\xi|<\de $ and $ \norm{\phi}^2 = \la $. Note that for such functions, $\p_x^{-1} \phi$ is well-defined.

	Denote the set of all such $ \phi$ as $ A_\la $, noting that $A_\la$  is dense in the set $\{u\in H^1: \|u\|_{L^2}^2=\la\}$. For such a $ \phi $
	\begin{equation}\label{equiv1}
		\ci[\phi]  = I [\p_x^{-1}\phi]\geq m_I(\la).
	\end{equation}
	Taking the infimum over all $ \phi\in A_\la $ gives us $ m_\ci(\la)\geq m_I(\la) $. On the other hand,
	\begin{equation*}%\label{}
		m_\ci(\la) = \inf_{u\in A_\la}\ci[u]\leq\inf_{u\in A_\la, u= v_x\in H^2}\ci[u] = m_I(\la).
	\end{equation*}
	So, $ m_I(\la) = m_\ci(\la) $.  Now, suppose $ \varphi_\la $ is a minimizer for \eqref{minI}, then, clearly, for $ \phi_\la := \varphi_\la' $ we have
	$ I[\varphi_\la] = \ci[\phi_\la] $.
\end{proof}
\subsection{Minimizing sequences produce non-trivial limits}

Now that we know that the minimization problems with cost functions $m_I$ and $m_\ci$ are equivalent,  suppose $\{u_k\}_{k=1}^{\infty } $ is a minimizing for $\ci$, subject to the constraint $\norm{u}_{L^2}^2=\la$.  That is
\begin{equation}
	\label{70}
	\lim\limits_{k\rightarrow\infty }\ci[u_k] = m_I,\quad \norm{u_k}_{L^2}^2=\lambda
\end{equation}
(similarly for $ J $). Clearly, there exists a subsequence, renamed to $\{u_k\} _{k=1}^{\infty }$, such that
\begin{equation}
	\label{80}
	\int_{\rone}|\p_x u_k|^2 dx \rightarrow I_1,\quad \int_{\rone}|\p_x^{-1} u_k|^2 dx \rightarrow I_2, \quad \int_{\rone}|u_k|^{p} u_k dx \rightarrow I_3,
\end{equation}
or
\begin{equation}
	\label{90}
	\int_{\rone}|\p_x u_k|^2 dx \rightarrow J_1,\quad \int_{\rone}|\p_x^{-1} u_k|^2 dx \rightarrow J_2, \quad \int_{\rone}|u_k|^{p+1} dx \rightarrow J_3,
\end{equation}
for $\cj$. We have the following key lemma, that shows that such minimizing sequences can  not possibly be trivially converging to zero.
\begin{lemma}
	\label{nonv}
	For any minimizing sequence satisfying \eqref{80}(\eqref{90} respectively),
	\begin{itemize}
		\item[i)]  $ J_3> 0 $ for  $ 1<p<5 $.
		\item [ii)] $I_3> 0 $ for  $ 1<p<3 $.
	\end{itemize}
\end{lemma}
\begin{proof}
	First of all, clearly, $ I_3\geq 0,  J_3\geq 0 $.  Let $\la>0$. We need to show that strict inequality holds in both cases. We treat them  separately. \\
	\\
	{\bf Proof of $J_3>0$} Suppose for contradiction that $ J_3  = 0 $. Then we can estimate the  infimum explicitly
	\begin{align}\label{211}
		\begin{split}
			m_\cj(\la)& = \inf_{\norm{u}_2^2=\lambda}\{\frac{1}{2} \int_\rone  |u_x|^2 +|\p_x^{-1} u|^2dx \}\\
			& =\inf_{\norm{u}_2^2=\lambda}\{\frac{1}{2}\int_\rone  \frac{ (\xi^2 - 1 )^2}{\xi^2}|\widehat{u}(\xi )|^2d\xi +\int_{\rone}|\widehat{u}(\xi)|^2d\xi\}\geq  \lambda.
		\end{split}
	\end{align}
	In fact, there is an equality above, as it suffices to take a function, whose Fourier transform is   highly localized around say $\xi=1$. The point is that this infimum is actually strictly smaller than $\la$, which would give us  the contradiction sought in this case.

	To see this, let $\chi_1$ be a Schwartz function, whose Fourier transform $\widehat{\chi}_1$ is  an  even bump $C^\infty$ function, supported  in the interval $ (-\f{1}{100}, \f{1}{100})$. Consider then $\chi:=\chi_1^2$, so that $\widehat{\chi}=\widehat{\chi_1}*\widehat{\chi_1}$.
	It has essentially the same properties as $\chi_1$, except it is in addition a positive function. That is, $\chi\geq 0$ and $supp\  \widehat{\chi}\subset  (-\f{1}{50}, \f{1}{50})$. Multiplication by a constant will help us  to achieve
	$ \norm{\chi}^2_2=\la/2$, which we assume henceforth.

	Next, consider the function
	\begin{equation*}%\label{}
		\widehat{v_{J,\varepsilon}}(\xi) =\frac{1}{\sqrt{\varepsilon}}\left ( \widehat{\chi}(\frac{\xi -1 }{\varepsilon}) + \widehat{\chi}(\frac{\xi +1}{\varepsilon})\right ),
	\end{equation*}
	By the support properties of $\chi$ and $\norm{\chi}^2_2=\la/2$, it is clear that for small $\eps$,  $\|v_{J,\eps}\|_{L^2}^2=\la$. Since $ \widehat{\chi} $ is even, we have that the function
	\begin{equation*}%\label{}
		v_{J,\varepsilon}(x) = \sqrt{\varepsilon} \chi(\varepsilon x)(e^{ ix}+ e^{- ix}) = 2\sqrt{\varepsilon} \chi(\varepsilon x)\cos( x)
	\end{equation*}
	is  real.
	Next, using the fact that $ \widehat{\chi}(\frac{\xi-1}{\varepsilon}) $ and $ \widehat{\chi}(\frac{\xi+1}{\varepsilon}) $ have disjoint support and change of variables
	\begin{align}\label{est1}
		\begin{split}
			\frac{1}{2}\int_\rone  \frac{ (\xi^2 - 1)^2}{\xi^2}|\widehat{v_{J,\varepsilon}}(\xi )|^2d\xi
			&= \frac{1}{2}\int_\rone  \frac{ (\varepsilon\xi)^2(\varepsilon\xi + 2)^2}{(\varepsilon\xi+1)^2}|\widehat{\chi}(\xi )|^2d\xi + \frac{1}{2}\int_\rone  \frac{ (\varepsilon \xi - 2)^2(\varepsilon\xi)^2}{(\varepsilon\xi-1)^2}|\widehat{\chi}(\xi )|^2d\xi  \\
			&=O(\varepsilon^2)\\
		\end{split}
	\end{align}
	Note that the denominators above are never problematic, as they vanish away from the support of $\widehat{\chi}$.
	On the other hand, using lemma \ref{portion} and the non-negativity of $\chi$,  we get
	\begin{align}\label{est2}
		\begin{split}
			\int_{\rone}|v_{J,\varepsilon}(x)|^{p+1} dx
			&= 2^{p+1} \varepsilon ^ {\frac{p-1}{2}}\int_\rone \chi^{p+1}(x)|\cos\left (\frac{x}{\varepsilon}\right )|^{p+1}dx\\
			&\geq 2^{p+1} \varepsilon ^ {\frac{p-1}{2}} \frac{\sqrt{2}}{2}\sum_{n=-\infty }^{\infty} \int_{\varepsilon(2\pi n) }^{\varepsilon(2\pi n +\pi/4)}\chi^{p+1}(x)dx\\
			&\geq C \varepsilon^{\frac{p-1}{2}} \int \chi^{p+1}(x)dx+O(\eps^{\f{p+1}{2}}).
		\end{split}
	\end{align}
	Combining \eqref{est1} and \eqref{est2} we obtain
	\begin{equation*}%\label{}
		\cj[v_{J,\varepsilon}] =  O(\varepsilon^2) + \la - C \varepsilon ^{\frac{p-1}{2}},
	\end{equation*}
	which implies that for $p<5$, $ m_\cj < \lambda $ and  this is a contradiction with  \eqref{211}. Thus, $J_3>0$. \\
	\\
	{\bf Proof of $I_3>0$} The considerations in this case are considerably more involved.

	Similarly to \eqref{211}, we first establish  that $ m_\ci \geq  \lambda $ in this case.
	There is a slight twist that the quantity $\int_{\rone}|u|^{p} u dx$ is not necessarily non-negative anymore. However, since the other two quantities in the definition of $\ci$ are positive definite, we can (by switching $u\to -u$ if necessary) to assume that the infimum is taken over $u$, with the property
	$\int_{\rone}|u|^{p} u dx\geq 0$. This will give a better (i.e. smaller or equal)   $m_\ci$, which is what needs to happen anyway  as $m_\ci$ is the infimum. Then, it is clear that
	$$
		m_\ci(\la)\leq \inf_{\norm{u}_2^2=\lambda}\{\frac{1}{2} \int_\rone  |u_x|^2 +|\p_x^{-1} u|^2dx \}.
	$$
	On the other hand, our assumption that $I_3=0$, means that the opposite inequality also holds true as
	$$
		m_\ci=\lim_k \left(\frac{1}{2}  \int_\rone  |\p_x u_k|^2 +|\p_x^{-1} u_k|^2dx -
		\f{1}{p+1} \int_{\rone}|u_k|^{p} u_k dx\right)\geq  \inf_{\norm{u}_2^2=\lambda} \frac{1}{2} \int_\rone  |u_x|^2 +|\p_x^{-1} u|^2dx.
	$$
	This means, in particular  that $m_\ci\geq \la$, as we have argued before.  We will show that this is contradictory. To that end, consider
	\begin{equation*}%\label{}
		\widehat{v_{I,\varepsilon}}(\xi) =\frac{1}{\sqrt{\varepsilon}}\left( \widehat{\chi}(\frac{\xi -1 }{\varepsilon}) + \widehat{\chi}(\frac{\xi +1}{\varepsilon}) +\varepsilon^\alpha
		\bigg(\widehat{\chi}(\frac{\xi -2 }{\varepsilon}) + \widehat{\chi}(\frac{\xi +2}{\varepsilon})\bigg)  \right),
	\end{equation*}
	with $ \chi  $ as before  and $ \max(\frac{p-1}{2},\frac{2}{p+1})<\alpha<1 $. This is possible, due to the assumption $1<p<3$.  Note that $\al>\frac{2}{p+1}>\f{1}{2}$, due to the same assumption.  Then the function
	\begin{equation*}%\label{}
		v_{I,\varepsilon}(x) = \sqrt{\varepsilon} \chi(\varepsilon x)(e^{ ix}+ e^{- ix} +\varepsilon^{\alpha}(e^{ 2ix}+ e^{- 2ix})) = 2\sqrt{\varepsilon} \chi(\varepsilon x)(\cos( x) + \varepsilon^{\alpha} \cos(2x))
	\end{equation*}
	is  real and even. Similarly to \eqref{est1} we get
	\begin{equation*}%\label{}
		\int_\rone  \frac{\left (\xi^2 - 1\right )^2}{\xi^2}|\widehat{v_{I,\varepsilon}}(\xi )|^2d\xi= O(\varepsilon^{2\alpha})
	\end{equation*}
	for all $ \varepsilon $ small enough. Indeed, all terms in $V_{J, \eps}$  have disjoint Fourier support, due to the properties of $\chi$. However,  the dominant terms, due to the choice of $\al$,  are those with $\eps^\al$ in front of it, whence the bound $O(\varepsilon^{2\alpha})$.

	Now, we are going to show that
	\begin{equation}
		\label{100}
		\int_\rone |v_{I,\varepsilon}(x)|^p v_{I,\varepsilon}(x)dx \geq C \varepsilon^{\frac{p-1}{2}+\alpha},
	\end{equation}
	which will finish the proof of lemma, since $\f{p-1}{2}+\al<2\al$.

	To this end, let $ \gamma>0  $ be such that $ (p+1)(\frac{1}{2}-\gamma)=1$ (or $\ga:=\f{p-1}{2(p+1)}\in (0, \f{1}{2})$) and split the integral as follows
	\begin{align*}%\label{}
		 & \varepsilon ^{\frac{p+1}{2}}\int_\rone \chi^{p+1}(\varepsilon x)(\cos(x)+\varepsilon^\alpha \cos(2x))|\cos(x)+\varepsilon^\alpha \cos(2x)|^pdx                                      \\
		 & =\varepsilon^{\frac{p+1}{2}}\int_{|\cos(x)|\leq\varepsilon^{1/2-\gamma}}\chi^{p+1}(\varepsilon x)(\cos(x)+\varepsilon^{\alpha} \cos(2x))|\cos(x)+\varepsilon^{\alpha} \cos(2x)|^pdx \\
		 & + \varepsilon^{\frac{p+1}{2}}\int_{|\cos(x)|>\varepsilon^{1/2-\gamma}}\chi^{p+1}(\varepsilon x)(\cos(x)+\varepsilon^{\alpha} \cos(2x))|\cos(x)+\varepsilon^{\alpha} \cos(2x)|^pdx
	\end{align*}
	For the first term we have
	\begin{align*}
		 & \left |\varepsilon^{\frac{p+1}{2}}\int_{|\cos(x)|\leq\varepsilon^{1/2-\gamma}}\chi^{p+1}(\varepsilon x)(\cos(x)+\varepsilon^{\alpha} \cos(2x))|\cos(x)+\varepsilon^\alpha \cos(2x)|^pdx\right |                                             \\
		 & \leq \varepsilon^{\frac{p+1}{2}}\int_{\rone}\chi^{p+1}(\varepsilon x) \cdot \varepsilon^{\frac{1}{2}-\gamma }\cdot \varepsilon^{p(\frac{1}{2}-\gamma)}dx\leq C \varepsilon^{\frac{p-1}{2} + (p+1)(\frac{1}{2}-\gamma)}=C \eps^{\f{p+1}{2}}.
	\end{align*}
	as $ \varepsilon\rightarrow0+ $.

	Now we show that the second term is bounded below by $  C \varepsilon^{\frac{p-1}{2}+\alpha}$, and hence is dominant. In order to prepare the calculation, note that for $x: |\cos(x)|>\varepsilon^{1/2-\gamma}$, and $\eps<<1$,
	\begin{eqnarray*}
		|\cos(x)+\varepsilon^\al \cos(2x)|^p &=& |\cos(x)|^p\left (1+2\varepsilon^\al \frac{\cos(2x)}{\cos(x)}+\eps^{2\al} \frac{\cos^2(2x)}{\cos^2(x)}\right )^{p/2} = \\
		&=& |\cos(x)|^p \left(1+p\varepsilon^\al  \frac{\cos(2x)}{\cos(x)}\right)+O(\eps^{2\al+2\ga-1}),
	\end{eqnarray*}
	where in this calculation, we have implicitly used that $\al>\f{1}{2}, \ga>0$.

	Thus, inputting this expansion below,
	\begin{eqnarray*}
		& & \varepsilon^{\frac{p+1}{2}}\int\limits_{|\cos(x)|>\varepsilon^{1/2-\gamma}}\chi^{p+1}
		(\varepsilon x)(\cos(x)+\varepsilon^\al \cos(2x))|\cos(x)+\varepsilon^\al \cos(2x)|^pdx \\
		& = & \varepsilon^{\frac{p+1}{2}}\int\limits_{|\cos(x)|>\varepsilon^{1/2-\gamma}}\chi^{p+1}(\varepsilon x)(\cos(x)+\varepsilon^\alpha \cos(2x))|\cos(x)|^p\left (1+p\varepsilon^\al  \frac{\cos(2x)}{\cos(x)}\right )dx+ \\
		&+& O(\varepsilon^{\frac{p-1}{2}+(2 \al+ 2\gamma-1)})\\
		& = & \varepsilon^{\frac{p+1}{2}} \int\limits_{|\cos(x)|>\varepsilon^{1/2-\gamma}}\chi^{p+1}(\varepsilon x)\cos(x)|\cos(x)|^pdx +\\
		&+&  \varepsilon^{\alpha+\f{p+1}{2}}(p+1)\left(\int_{|\cos(x)|>\varepsilon^{1/2-\gamma}}\chi^{p+1}(\varepsilon x)\cos(2x)|\cos(x)|^pdx\right ) +  O(\varepsilon^{\frac{p-1}{2}+(2 \al+ 2\gamma-1)})\\
		&= &  \varepsilon^{\frac{p+1}{2}}\left (K + \varepsilon^{\al}(p+1)Q\right ) +O(\varepsilon^{\frac{p-1}{2} +(2 \al+ 2\gamma-1)})+O(\varepsilon^{\frac{p-1}{2} +2 \al}).
	\end{eqnarray*}
	where we have introduced two quantities $K,Q$. Clearly, since $2\ga<1$, the term $O(\varepsilon^{\frac{p-1}{2} +(2 \al+ 2\gamma-1)})$ is dominant over $O(\varepsilon^{\frac{p-1}{2} +2 \al})$.

	We claim that $ K=O(1)$, whereas $Q \geq C\eps^{-1}$.
	This  implies \eqref{100} and the the proof of Lemma \ref{nonv} will be complete.
	First, let us deal with $ K $.
	\begin{align*}
		K=\int_{|\cos(x)|>\varepsilon^{1/2-\gamma}}\chi^{p+1}(\varepsilon x)\cos(x)|\cos(x)|^pdx & = \int_{\rone}\chi^{p+1}(\varepsilon x)\cos(x)|\cos(x)|^pdx  + O(1).
	\end{align*}
	The change of variables $ y =  \pi/2-x $  yields
	$$
		\int_{\rone}\chi^{p+1}(\varepsilon x)\cos(x)|\cos(x)|^pdx
		=-\int_{\rone}\chi^{p+1}(\varepsilon (\pi/2-y))\sin(y)|\sin(y)|^p dx
	$$
	Observe however that for $ F(u)  := \int_{0}^{u} (z-z^2)^{p/2}dz$, $ 0<u<1 $, we have
	$$
		\sin(y)|\sin(y)|^p = 2^{p+1} \p_y[F(\sin^2(y/2))],
	$$
	and hence integrating by parts yields
	\begin{eqnarray*}
		\int_{\rone}\chi^{p+1}(\varepsilon x)\cos(x)|\cos(x)|^pdx = - 2^{p+1} \varepsilon \int_\rone \frac{\partial}{\p y}(\chi^{p+1})(\varepsilon (\pi/2-y)) F(\sin^2(y/2))dy
		=O(1),
	\end{eqnarray*}
	since $F$ is continuous function.

We now  prove the claim about $Q$. Similarly, to $K$ we can write

	\begin{align*}
		\int_{|\cos(x)|>\varepsilon^{1/2-\gamma}}\chi^{p+1}(\varepsilon x)\cos(2x)|\cos(x)|^pdx & = \int_{\rone}\chi^{p+1}(\varepsilon x)\cos(2x)|\cos(x)|^pdx  + O(\varepsilon^{p(\frac{1}{2}-\gamma)-1}).
	\end{align*}
	Noting  that $p(\frac{1}{2}-\gamma)-1>-1$, it suffices to show that the first term is bounded from below by $C \eps^{-1}$.

	Splitting each of the intervals $ [2\pi n, 2\pi (n+1))$ into eight  pieces as follows
	\begin{eqnarray*}
		& &\int_{\rone}\chi^{p+1}(\varepsilon x)\cos(2x)|\cos(x)|^pdx \\
		&=& \left (\sum_{n=-\infty}^{\infty}\int_{2 \pi n}^{2\pi n +\frac{\pi}{4}}\chi^{p+1}(\varepsilon x)\cos(2x)|\cos(x)|^pdx + \sum_{n=-\infty}^{\infty}\int_{2 \pi n +\frac{\pi}{2}}^{2\pi n+\frac{3\pi}{4}}\chi^{p+1}(\varepsilon x)\cos(2x)|\cos(x)|^pdx  \right )\\
		&+& \left (\sum_{n=-\infty}^{\infty}\int_{2 \pi n + \frac{3\pi}{4}}^{2\pi n +\pi}\chi^{p+1}(\varepsilon x)\cos(2x)|\cos(x)|^pdx + \sum_{n=-\infty}^{\infty}\int_{2 \pi n +\frac{\pi}{4}}^{2\pi n+\frac{\pi}{2}}\chi^{p+1}(\varepsilon x)\cos(2x)|\cos(x)|^pdx  \right )\\
		&+& \left (\sum_{n=-\infty}^{\infty}\int_{2 \pi n +\pi}^{2\pi n +\frac{5\pi}{4}}\chi^{p+1}(\varepsilon x)\cos(2x)|\cos(x)|^pdx + \sum_{n=-\infty}^{\infty}\int_{2 \pi n +\frac{3\pi}{2}}^{2\pi n+\frac{7\pi}{4}}\chi^{p+1}(\varepsilon x)\cos(2x)|\cos(x)|^pdx  \right )\\
		&+ & \left (\sum_{n=-\infty}^{\infty}\int_{2 \pi n +\frac{7\pi}{4}}^{2\pi n+2\pi}\chi^{p+1}(\varepsilon x)\cos(2x)|\cos(x)|^pdx+\sum_{n=-\infty}^{\infty}\int_{2 \pi n +\frac{5\pi}{4}}^{2\pi n +\frac{3\pi}{2}}\chi^{p+1}(\varepsilon x)\cos(2x)|\cos(x)|^pdx    \right )
	\end{eqnarray*}
	and then pairing them as in \figurename{ \ref{fig:circle}}
	\begin{figure}
		\centering
		\includegraphics[width=0.7\linewidth]{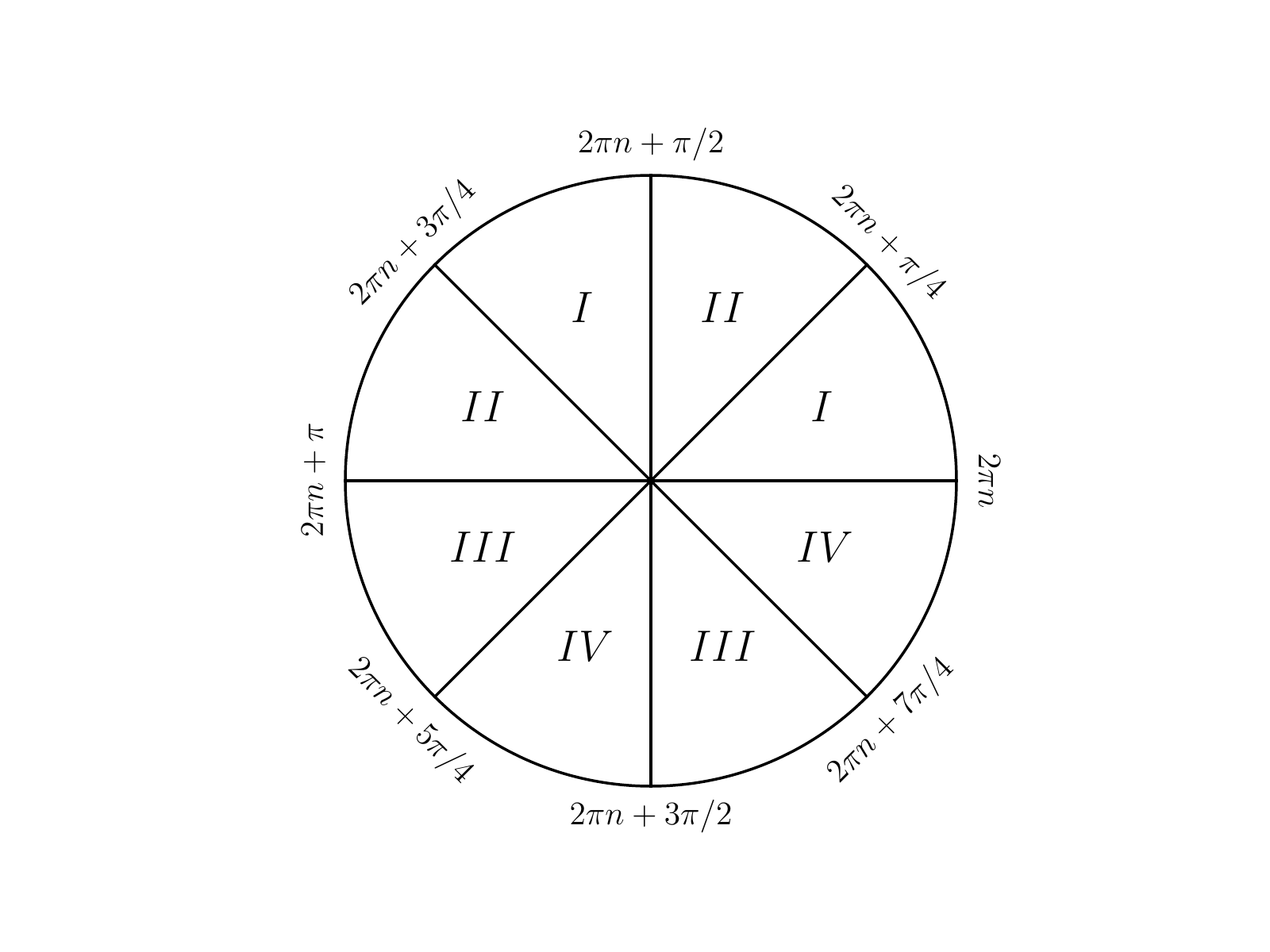}
		\caption{}
		\label{fig:circle}
	\end{figure}
	yields
	\begin{align*}
		 & \sum_{n=-\infty}^{\infty}\int_{2\pi n}^{2\pi n +\frac{\pi}{4}}\chi^{p+1}(\varepsilon x)\cos(2x)(|\cos(x)|^p-|\sin(x)|^p)dx                                                                 \\
		 & + \sum_{n=-\infty}^{\infty}\int_{2\pi n+\frac{3\pi}{4}}^{2\pi n +\pi}\chi^{p+1}(\varepsilon x)\cos(2x)(|\cos(x)|^p-|\sin(x)|^p)dx                                                          \\
		 & + \sum_{n=-\infty}^{\infty}\int_{2\pi n+\pi}^{2\pi n +\frac{5\pi}{4}}\chi^{p+1}(\varepsilon x)\cos(2x)(|\cos(x)|^p-|\sin(x)|^p)dx                                                          \\
		 & + \sum_{n=-\infty}^{\infty}\int_{2\pi n+\frac{7\pi}{4}}^{2\pi n +2\pi}\chi^{p+1}(\varepsilon x)\cos(2x)(|\cos(x)|^p-|\sin(x)|^p)dx                                                         \\
		 & + \sum_{n=-\infty}^{\infty}\int_{2\pi n}^{2\pi n+\frac{\pi}{4}}\left (\chi^{p+1}\left (\varepsilon (x+\frac{\pi}{2})\right)-\chi^{p+1}(\varepsilon x)\right )\cos(2x)|\sin(x)|^pdx         \\
		 & + \sum_{n=-\infty}^{\infty}\int_{2\pi n+\frac{3\pi}{4}}^{2\pi n +\pi}\left (\chi^{p+1}\left (\varepsilon (x-\frac{\pi}{2})\right)-\chi^{p+1}(\varepsilon x)\right )\cos(2x)|\sin(x)|^pdx   \\
		 & + \sum_{n=-\infty}^{\infty}\int_{2\pi n+\pi}^{2\pi n +\frac{5\pi}{4}}\left (\chi^{p+1}\left (\varepsilon (x+\frac{\pi}{2})\right)-\chi^{p+1}(\varepsilon x)\right )\cos(2x)|\sin(x)|^pdx   \\
		 & + \sum_{n=-\infty}^{\infty}\int_{2\pi n+\frac{7\pi}{4}}^{2\pi n +2\pi}\left (\chi^{p+1}\left (\varepsilon (x-\frac{\pi}{2})\right)-\chi^{p+1}(\varepsilon x)\right )\cos(2x)|\sin(x)|^pdx. \\
		%&\geq C \sum_{n=-\infty}^{\infty}\int_{2\pi n+\frac{\pi}{12}}^{2\pi n +\frac{\pi}{4}-\frac{\pi}{12}}\chi^{p+1}(\varepsilon x)dx+ O(1)\\
		%&\ge \frac{C}{\varepsilon} 
	\end{align*}
	Note that the first four terms are all positive for all values of $n$.
	In addition, taking the first term
	\begin{eqnarray*}
		& &  \sum_{n=-\infty}^{\infty}\int_{2\pi n}^{2\pi n +\frac{\pi}{4}}\chi^{p+1}(\varepsilon x)\cos(2x)(|\cos(x)|^p-|\sin(x)|^p)dx \geq  \\
		&\geq &  \sum_{n=-\infty}^{\infty}\int_{2\pi n}^{2\pi n +\frac{\pi}{6}}\chi^{p+1}(\varepsilon x)\cos(2x)(|\cos(x)|^p-|\sin(x)|^p)dx \geq c_p \sum_{n=-\infty}^{\infty}\int_{2\pi n}^{2\pi n +\frac{\pi}{6}}\chi^{p+1}(\varepsilon x) dx \\
		& \geq & \eps^{-1} \sum_{n=-\infty}^{\infty}  \int_{2\pi n \eps}^{2\pi n\eps +\frac{\pi}{6}\eps}\chi^{p+1}(y) dy  \geq d_p \eps^{-1} \int_{\rone} \chi^{p+1}(y) dy dy+O(1).
	\end{eqnarray*}
	by Lemma \ref{portion}.

	On the other hand, for the error terms we have a bound of $O(1)$, since
	\begin{eqnarray*}
		& & \sum_{n=-\infty}^{\infty}|\int_{2\pi n}^{2\pi n+\frac{\pi}{4}}\left (\chi^{p+1}\left (\varepsilon (x+\frac{\pi}{2})\right)-\chi^{p+1}(\varepsilon x)\right )\cos(2x)|\sin(x)|^pdx|\leq \\
		&\leq&  \sum_{n=-\infty}^{\infty}\int_{2\pi n}^{2\pi n+\frac{\pi}{4}}\int_{\varepsilon x}^{\varepsilon(x+\frac{\pi}{2})}|(\chi^{p+1})'(y)|dy dx \leq \sum_{n=-\infty}^{\infty}\int_{2\pi n}^{2\pi n+\frac{\pi}{4}}\int_{\varepsilon 2\pi n}^{\varepsilon(2\pi n+\frac{3\pi}{4})}|(\chi^{p+1})'(y)|dy dx \\
		&\leq & C \int_\rone |(\chi^{p+1})'\left (x\right )|dx
	\end{eqnarray*}
	and, similarly,  we estimate the  three  other error terms.
\end{proof}
\noindent We are now ready to present  the main result of this section.
\subsection{Existence of the waves}
\begin{proposition}
	\label{prop:20}
	Let $1<p<3$. Then, the minimization problem \eqref{minI} has a solution. For $1<p<5$, the minimization problem \eqref{minJ} has a solution.
\end{proposition}
{\bf Remark:} By Lemma \ref{le:3}, this implies the existence of solutions to \eqref{minIp} and \eqref{minJp}, in the corresponding range of $p$.
The proof of Proposition \ref{prop:20} is based on the method of concentrated compactness. In the compensation compactness arguments, the sub-additivity of the function $\la\to m(\la)$ plays a pivotal role. We begin with this  lemma. Note that the technical Lemma \ref{nonv} is needed precisely in this step.
\begin{lemma}(Strict sub-additivity holds)
	\label{le:40}
	Fix  $ \lambda >0$.
	\begin{enumerate}%[a)]
		\item[i)] Suppose $ 1<p<3 $ . Then for all $ 0<\alpha<\lambda $  we have that the strict sub-additivity condition holds for $ m_I $, namely,
		      \begin{equation*}%\label{}
			      m_I(\lambda)< m_I(\alpha)+m_I(\lambda-\alpha).
		      \end{equation*}
		\item[ii)] Suppose $ 1<p<5 $ . Then for all $ 0<\alpha<\lambda $  we have that the strict sub-additivity condition holds for $ m_J $, namely,
		      \begin{equation*}%\label{}
			      m_J(\lambda)< m_J(\alpha)+m_J(\lambda-\alpha).
		      \end{equation*}
	\end{enumerate}
\end{lemma}
\begin{proof}
	The proofs of i) and ii) are identical. Let us prove i). First, we claim that the function $ \frac{m_I(\la)}{\la} $ is strictly decreasing. Indeed,
	\begin{align*}%\label{}
		m_I(\la) & = \inf_{\norm{u_x}_2^2=\lambda}\{\frac{1}{2}\int_{\rone}|u_{xx}|^2+|u|^2dx -\frac{1}{p+1}\int_{\rone}|u_x|^p u_x dx\}                                  \\
		         & = \frac{\la}{\al}\inf_{\norm{u_x}_2^2=\al}\{\frac{1}{2}\int_{\rone}|u_{xx}|^2+|u|^2dx -\frac{(\la/\al)^{\frac{p-1}{2}}}{p+1}\int_{\rone}|u_x|^p u_x dx \\
		         & < \frac{\la}{\al} m_I(\al),
	\end{align*}
	where the strict inequality follows from the fact that by lemma \ref{nonv} there exist a minimizing sequence $ \{u_k\}_{k=1}^{\infty} $ such that $ \lim_{k\rightarrow \infty }\int_\rone |(u_k)_x|^p (u_k)_x dx >0 $. Finally, assuming that $ \al\in [\la/2,\la) $ (otherwise we argue  with $ \la-\al $), since $ \frac{m_I(\la)}{\la} $ is decreasing,  we get
	\begin{equation*}%\label{}
		m_I(\la) <\frac{\la}{\al}m_I(\al)= m_I(\al) + \frac{\la-\al}{\al}m_I(\al)\leq m_I(\al)+ m_I(\la-\al).
	\end{equation*}
\end{proof}
\begin{proof}(Proposition \ref{prop:20})

	Define $ \rho_k(x) = |\p_x u_k|^2 $. By the concentration compactness lemma at least one of the following holds:
	\begin{itemize}
		\item [i)] \textit{Tightness.} There exists $ \{y_k\}_{k=1}^{\infty} $ such that for all $ \varepsilon>0 $ there exists an $ R_\varepsilon>0 $ satisfying
		      \begin{equation*}%\label{}
			      \int_{B(y_k,R_\varepsilon )}\rho_k dx \geq \int_\rone \rho_k dx -\varepsilon
		      \end{equation*}
		\item [ii)] \textit{Vanishing.} For every $ R>0 $
		      \begin{equation*}%\label{}
			      \lim_{k\rightarrow \infty}\sup_{y\in \rone}\int_{B(y,R)}\rho_k dx = 0.
		      \end{equation*}
		\item [iii)] \textit{Dichotomy}. There exists an $ \alpha\in (0,\la) $ such that for every $ \varepsilon>0 $ there exist $ R $, $ R_k\rightarrow\infty $, $ y_k $  and $ k_0 $ such that for all $ k\geq k_0 $
		      \begin{equation*}%\label{}
			      \left | \int_{|x-y_k|<R}\rho_k dx  - \al \right | <\varepsilon, \left | \int_{|x-y_k|>R_k}\rho_k dx  - (\la-\al) \right | <\varepsilon, \left | \int_{R<|x-y_k|<R_k}\rho_k dx  \right | <\varepsilon.
		      \end{equation*}
	\end{itemize}
	\subsubsection{Vanishing cannot occur} 	First,  suppose vanishing occurs. Let
	$ 0\leq\chi\leq 1 $ be  a smooth bump function  supported on $ (-2,2) $ with $ \chi \equiv 1$ on $  (-1,1) $.   Applying GNS iequality we get
	\begin{align}%\label{}
		\begin{split}
			\int_{B(y,1)}|(u_k)_x|^p(u_k)_xdx
			&\leq \int_{\rone}|(u_k)_x\chi(x-y)|^{p+1}dx\\
			&\leq \norm{(u_k)_x\chi(x-y)}_2^{\frac{p+3}{2}} \norm{((u_k)_x\chi(x-y))_x}_2^{\frac{p-1}{2}}\\
			&\leq C \norm{(u_k)_x}_{L^2(B(y,2))}^{\frac{p+3}{2}},
			\label{pplusbound}
		\end{split}
	\end{align}
	where in the last line, we have used that $\|u_k\|_{H^2}$ is a bounded sequence.
	By the assumed vanishing, choose $ k_0  $ so large that for all $ k\geq k_0 $
	\begin{equation*}%\label{}
		\int_{B(y,2)}\rho_k dx <\varepsilon
	\end{equation*}
	for all $ y\in R $. We can cover the real line with intervals
	$ \cup_{n=0}^{\infty}B(y_n,2) $ so that each $ x\in \rone  $ belongs to at most ten  intervals and  $ \cup_{n=0}^{\infty}B(y_n,1) $ still covers the whole line. Using \eqref{pplusbound}, we obtain
	\begin{align*}%\label{}
		\int_{\rone}|(u_k)_x|^{p} (u_k)_xdx
		 & \leq\int_{\rone}|(u_k)_x|^{p+1}dx\leq \sum_{n=0}^{\infty}\int_{B(y_n,1)}|(u_k)_x|^{p+1} dx                                                         \\
		 & \leq C\varepsilon^{\frac{p-1}{2}} \sum_{n=0}^{\infty}  \norm{(u_k)_x}^2_{L^2{B(y_n,2)}}\leq 3 C \varepsilon^{\frac{p-1}{2}} \norm{(u_k)_x}^2_{L^2}
	\end{align*}
	which  is a contradiction, for sufficiently small $\eps>0$. Indeed, recall that $\sup_k \|u_k\|_{H^2}<\infty$, while  by Lemma \ref{nonv}, $\inf_k 	\int_{\rone}|(u_k)_x|^{p} (u_k)_xdx>0$.   Hence, vanishing cannot occur.
	\subsubsection{Dichotomy cannot occur}
	Suppose dichotomy occurs.  Let  $ \eta_1,\eta_2 \in C^\infty(\rone) $, satisfying $ 0\leq \eta_1, \eta_2\leq 1 $ and
	\begin{equation*}%\label{}
		\eta_1(x) = \begin{cases}
			1, & |x|\leq1,  \\
			0, & |x|\geq 2, \\
		\end{cases}
		,
		\quad
		\eta_2(x)= \begin{cases}
			1, & |x|\geq 1,   \\
			0, & |x|\leq 1/2. \\
		\end{cases}
	\end{equation*}
	Dichotomy implies that there exists a subsequence of $ \{u_k\}_{k=1}^{\infty} $ (re-indexed to be $ \{u_k\}_{k=1}^{\infty} $ again) and  sequences $ \{R_k\}_{k=1}^{\infty} \in\rone$, with $ \lim_{k\rightarrow\infty}R_k =\infty$
	and $ \{y_k\}_{k=1}^{\infty} \in \rone$ such that

	\begin{equation*}%\label{}
		\lim_{k\rightarrow\infty}\int_\rone |(u_{k,1})_x|^2 dx  = \alpha,\quad
		\lim_{k\rightarrow\infty}\int_\rone |(u_{k,2})_x|^2 dx  = \lambda - \alpha,\quad
		\int_{R_k/5\leq |x-y_k|<R_k} |(u_{k})_x|^2 dx  \leq \frac{1}{k},\quad
	\end{equation*}
	where
	\begin{equation*}%\label{}
		u_{k,1}(x) = u_k(x)\eta_1\left (\frac{x-y_k}{R_k/5}\right ),\quad
		u_{k,2}(x) = u_k(x)\eta_2\left (\frac{x-y_k}{R_k}\right ).
	\end{equation*}
	Let $ \{a_k\}_{k=1}^{\infty} $ and $ \{b_k\}_{k=1}^{\infty} $ be sequences of real numbers converging to $ 1 $ such that
	\begin{equation*}%\label{}
		\int_{\rone}|(a_ku_{k,1})_x|^2 dx=\al,\quad  \int_{\rone}|(b_ku_{k,1})_x|^2 dx=\la-\alpha,
	\end{equation*}
	for all $ k $. It is easy to see that the following holds
	\begin{align*}\label{difference}
		\begin{split}
			I[u_k]-I[a_k u_{k,1}] - I[b_k u_{k,1}]&= \frac{1}{2}\int_{\rone}\left (1-\eta_1^2\left (\frac{x-y_k}{R_k/5}\right )-\eta^2_2\left (\frac{x-y_k}{R_k}\right )\right )
			\left( |(u_k)_{xx}|^2+|u_k|^2\right)dx \\
			&+ O\left (\frac{1}{R_k}\right ) +O\left (\frac{1}{k}\right ) + O\left (|1-a_k^2|\right ) + O\left (|1-b_k^2|\right ).\\
		\end{split}
	\end{align*}
	It follows that
	\begin{equation*}
		I[u_k] \geq  m_I(\alpha) + m_I(\la -\al) +\be_k,
	\end{equation*}
	where $ \be_k \rightarrow 0 $. Taking the limit as $ k\rightarrow \infty $  we obtain
	\begin{equation*}%\label{}
		m_I(\la) \geq m_I(\al) + m_I(\la-\al),
	\end{equation*}
	which contradicts  the strict sub-additivity condition shown in  lemma \ref{40}. Hence dichotomy is not an option.
	\subsubsection{Tightness implies existence of a minimizer}
	Now, using tightness we show existence of a minimizer. We show it only for the $ I $ functional, but the steps for the $ J $ functional are  exactly the same, if not easier.
	Define $ v_k(x) = u_k(x-y_k) $. Since $ \{v_k\}_{k=1}^{\infty} $ is bounded on $ H^2(\rone) $ there exists a weakly convergent subsequence to some $ v\in H^2(\rone) $, renamed to $ \{v_k\}_{k=1}^{\infty} $ again. From tightness it follows that   for all $ \varepsilon >0 $ there exists an $ R_\varepsilon $  satisfying
	\begin{equation}\label{tightness}
		\int_{B^c(0,R_\varepsilon)}|(v_k)_x|^2dx <\varepsilon.
	\end{equation}
	By the Rellich-Kondrachov theorem $ H^1(B(0,R_\varepsilon)) $  compactly embeds into $ L^2(B(0,R_\varepsilon)) $. So, there exists a subsequence of $ \{v_k\}_{k=1}^{\infty} $  such that $ (v_k)_x \rightarrow v_x $ strongly on $ L^2(B(0,R_\varepsilon)) $. Taking $ \varepsilon = 1/n $ and letting $ n\rightarrow\infty $  in \eqref{tightness} we can find a subsequence of $ \{v_k\}_{k=1}^{\infty} $, again renamed to be the same, so that $ (v_k)_x  \rightarrow v_x$ strongly on $ L^2(\rone) $. With this in hand, we can show that
	\begin{equation}\label{312}
		\lim_{k\rightarrow \infty }\int_{\rone}|(v_k)_x|^p(v_k)_xdx=\int_{\rone}|v_x|^pv_xdx.
	\end{equation}
	Indeed,
	\begin{align*}%\label{}
		\left |\int_{\rone}|(v_k)_x|^p(v_k)_xdx-\int_{\rone}|v_x|^pv_x dx\right |
		 & \leq C \int_{\rone} |(v_k)_x -v_k|(|(v_k)_x|^p+|v_x|^p)dx                                \\
		 & \leq C \norm{(v_k)_x-v_x}_{2}\left (\norm{v_x}_2+\norm{(v_k)_x}_2\right )^p\rightarrow0,
	\end{align*}
	where we have used the inequality $ ||x|^px- |y|^py| \leq C |x-y|\left (|x|^p+|y|^p\right ) $ which holds for all real numbers $ x$ and $ y $, the Cauchy-Schwartz inequality and the fact that $ H^1(\rone) $ embeds into $ L^\infty(\rone) $ (so that the sequence $ \{(v_k)_x\} _{k=1}^{\infty}$ is  bounded in $ L^\infty (\rone)$).

	Finally, the lower semi-continuity of norms with respect to weak convergence and \eqref{312} imply that $ m_I(\la)=\lim_{k\rightarrow\infty} I[v_k] \geq I[v]  $, which means that $ I[v] = m_I(\lambda) $ and $ v $ is the minimizer. Proposition \ref{prop:20} is thus proved in full.
\end{proof}
The next order of business is to derive the Euler-Lagrange equations.
\subsection{The Euler-Lagrange equations - fourth order formulations}
\begin{proposition}
	\label{prop:25}
	\quad
	\begin{itemize}
		\item For $ 1< p < 3 $ and $ \la > 0 $ there exists a real number  $ \omega$ such that the minimizer of the constrained minimization problem \eqref{minI}  $ \phi_\la $ satisfies the Euler-Lagrange equation
		      \begin{equation}
			      \label{ELI}
			      \phi_\la'''' + \om(\la)\phi_\la''+\phi_\la + \partial_x(|\phi_\la'|^p) =0,
		      \end{equation}
		      where $ \om=\omega(\lambda, \phi_\la) = \frac{1}{\lambda}\int_\rone[|\phi_\lambda''|^2 +|\phi_\lambda|^2-|\phi_\lambda'|^p\phi_\lambda']dx $. 
		\item For $ 1< p < 5 $ and $ \la > 0  $ there exists a function $ \omega$ such that the minimizer of the constrained minimization problem \eqref{minJ}  $ \phi_\la $ satisfies the Euler-Lagrange equation
		      \begin{equation}
			      \label{ELJ}
			      \phi_\la'''' + \om(\la)\phi_\la''+\phi_\la + \partial_x(|\phi_\la'|^{p-1}\phi_\la') =0,
		      \end{equation}
		      where $ \om=\omega(\lambda, \phi_\la)\ = \frac{1}{\lambda}\int_\rone|\phi_\lambda''|^2 +|\phi_\lambda|^2-|\phi_\lambda'|^{p+1}dx $.
	\end{itemize}
\end{proposition}
\begin{proof}
	Consider $ u_\delta =  \frac{\phi_\lambda + \delta h}{\norm{\phi_\lambda' + \delta h'}}\sqrt{\lambda}$, where $h$ is a test function.
	Clearly, $u_\delta$ satisfies the constraint and expanding $I[u_\delta]$ in $\delta$ we get
	$$	I[u_\delta] = m_I(\lambda)  + \delta \left( \int_\rone \phi_\lambda''h''dx +   \phi_\lambda h + |\phi_\lambda'|^ph' - \frac{1}{\lambda}(\int_{\rone}  |\phi_\lambda''|^2 +|\phi_\lambda|^2-|\phi_\lambda'|^p\phi_\lambda'dx) \int_{\rone}\phi_\la'h'dx\right)\\
		+	O(\de^2).
	$$
	Since $ I[u_\delta] \geq m_I[\la]$ for all $ \de\in \rone $    we  conclude that

	\begin{equation*}%\label{}
		\dpr{ \phi_\lambda''''+\phi_\la  +\omega(\lambda) \phi_\lambda'' + (|\phi_\lambda'|^p)'}{h}=0
	\end{equation*}
	with $ \omega = \frac{1}{\lambda}\int_\rone|\phi_\lambda''|^2 +|\phi_\lambda|^2-|\phi_\lambda'|^p\phi_\lambda'dx $, holds for all $ h $, i.e.,   $ \phi_\la $ is a distributional solution of the Euler-Lagrange Equation \eqref{ELI}.
	For the minimizers of \eqref{minJ}, we proceed analogously to establish \eqref{ELJ}.
\end{proof}
\subsection{The Euler-Lagrange equations - second  order formulation}
\begin{proposition}
	\label{prop:30}
	\begin{itemize}
		\item For $ 1< p < 3 $, there exists a function $ \omega (\la)$ such that for all $\la>0$, the minimizer of the constrained minimization problem \eqref{minIp}  $ \phi_\la $ satisfies the Euler-Lagrange equation
		      \begin{equation}
			      \label{347}
			      \p_x^2\phi_\la +\p_x^{-2}\phi_\la+ \om(\la)\phi_\la +|\phi_\la|^p =0,
		      \end{equation}
		      where
		      \begin{equation}
			      \label{3477}
			     \om= \omega(\lambda, \phi_\la) = \frac{1}{\lambda}\int_\rone[|\p_x\phi_\lambda|^2 +|\p_x^{-1}\phi_\lambda|^2-|\phi_\lambda|^p\phi_\lambda] dx.
		      \end{equation}
		      In addition,  the linearized operator $ \cl_+ :=-\p_x^2 -\p_x^{-2} -\om(\la) - p|\phi_\la|^{p-2}\phi_\la$ satisfies $\cl_+|_{\{\phi_\la\}^\perp}\geq 0$. In fact, $\cl_+$  has exactly
		      one negative eigenvalue.
		\item For $ 1< p < 5 $, there is $ \omega (\la)$, such that for all $\la>0$, the minimizer of the constrained minimization problem \eqref{minJp}  $ \phi_\la $ satisfies\footnote{in a weak sense}  the Euler-Lagrange equation
		      \begin{equation}
			      \label{352}
			      \p_x^{2}\phi_\la +\p_x^{-2}\phi_\la+ \om(\la)\phi_\la +|\phi_\la|^{p-1}\phi_\la =0,
		      \end{equation}
		      where $ \om=\omega(\lambda, \phi_\la) = \frac{1}{\lambda}\int_\rone|\phi_\lambda'|^2 +|\p_x^{-1}\phi_\lambda|^2-|\phi_\lambda|^{p+1}dx $.  The operator $ \cl_+  =-\p_x^2 -\p_x^{-2} -\om(\la) - p|\phi_\la|^{p-1}$ has $\cl_+|_{\{\phi_\la\}^\perp}\geq 0$ and it possesses exactly
		      one negative eigenvalue.
	\end{itemize}
\end{proposition}
\begin{proof}
	The derivation of the Euler-Lagrange equations is  pretty similar to the one presented in the fourth order context, Proposition \ref{prop:25}.
	For an arbitrary test function $h$ and $\de\in \rone$, consider  $ u_\delta = \sqrt{\lambda} \frac{\phi_\lambda + \delta h}{\norm{\phi_\lambda + \delta h}}$. Since $u_\de$   satisfies the constraint $\|u_{\de}\|_{L^2}^2=\la$,   expand $I[u_\delta]$ in powers of $\delta$.   In principle, in this calculation, we need to only account for up to linear terms and collect all else in $O(\de^2)$. However, due to a future need to understand the precise form of terms quadratic in $\de$, we include them as well. We get
	\begin{align*}
	&	\ci[u_{\delta}]  =m_{\ci}(\lambda)                                                                                                                                                                                                                                                                                            \\
		                & +\delta\left(\int_{\rone}\phi_{\lambda}'h'dx+\int_{\rone}\partial_{x}^{-1}\phi_{\lambda}\partial_{x}^{-1}hdx-\int_{\rone}|\phi_{\lambda}|^{p}h-\frac{1}{\lambda}\int_{\rone}|\phi_{\lambda}'|^{2}+|\partial_{x}^{-1}\phi_{\lambda}|^{2}-|\phi_{\lambda}|^{p}\phi_{\lambda}dx\int_{\rone}\phi_{\la}hdx\right) \\
		                & +\frac{\delta^{2}}{2}\left(\int_{\rone}|h'|^{2}+|\partial_{x}^{-1}h|^{2}dx-p\int_{\rone}|\phi_{\lambda}|^{p-2}\phi_{\lambda}|h|^{2}dx\right)                                                                                                                                                                 \\
		                & -\frac{\delta^{2}}{2}\frac{1}{\lambda}\left(\int_{\rone}|\phi_{\la}'|^{2}+|\partial_{x}^{-1}\phi_{\lambda}|^{2}-|\phi_{\lambda}|^{p}\phi_{\la}dx\right)\int_{\rone}|h|^{2}dx                                                                                                                                 \\
		                & +\delta^{2}\frac{\int_{\rone}\phi_{\lambda}hdx}{\lambda}\left((p+1)\int_{\rone}|\phi_{\la}|^{p}hdx-2(\int_{\rone}\phi_{\lambda}'h'dx+\int_{\rone}\partial_{x}^{-1}\phi_{\lambda}\partial_{x}^{-1}hdx)\right)                                                                                                 \\
		                & +\delta^{2}\left(\frac{\int_{\rone}\phi_{\lambda}hdx}{\lambda}\right)^{2}\left(2\int_{\rone}|\phi_{\lambda}'|^{2}+|\partial_{x}^{-1}\phi_{\lambda}|^{2}dx-\frac{p+3}{2}\int|\phi_{\la}|^{p}\phi_{\lambda}dx\right)    + O(\de^3).                                                                          
	\end{align*}
	Since $ \ci[u_\delta] \geq m_\ci[\la]$ for all $ \de\in \rone $,     we  conclude that
	\begin{equation}
		\label{348}
		\left\langle \phi_{\lambda}''+\partial_{x}^{-2}\phi_{\la}+\omega \phi_{\lambda}+|\phi_{\lambda}|^{p},h\right\rangle = 0
	\end{equation}
	with $ \omega  = \frac{1}{\lambda}\int_\rone|\phi_\lambda'|^2 +|\p_x^{-1}\phi_\lambda|^2-|\phi_\lambda|^p\phi_\lambda dx $, holds for all $ h $.  That is    $ \phi_\la $ is a distributional solution of the Euler-Lagrange Equation. According to Proposition \ref{prop:nm}, this solution is in fact an element of $H^3$ and \eqref{347} is satisfied in the sense of $L^2$ functions. 

	The fact that $\phi_\la  $ is a minimizer also implies that the coefficient in front of $ \de^2 $ must be nonnegative. Choosing $ h $  orthogonal to $ \phi_\lambda $ with $ \norm{h} =1 $, we conclude that
	\begin{equation*}%\label{}
		\left\langle - h'' - \p_x^{-2}h - p|\phi_\la|^{p-2}\phi_\la h-\om h   ,h  \right\rangle \geq 0,
	\end{equation*}
	i.e., the operator $ \cl_+ = -\partial_x^2 -\partial_x^{-2}-\om(\la) - p|\phi_\la|^{p-2}\phi_\la$, satisfies
	$ \left\langle  \cl_+ h,h  \right\rangle \geq 0$ for all $h $ orthogonal to $ \phi_\la $ with $ \norm{h}=1 $, which implies that it has at most one negative eigenvalue.  On the other hand, recalling that $\int_\rone |\phi_\la|^p\phi_\la dx>0$, we compute
	\begin{equation}
	\label{lplus}
		\left\langle  \cl_+  \phi_\la ,\phi_\la  \right\rangle = -(p-1)\int_\rone |\phi_\la|^p\phi_\la dx <0.
	\end{equation}
	So, $\cl_+$ has at least one negative eigenvalue.  Hence it has exactly one negative eigenvalue.
	The second part of the proposition is proven similarly expanding $ \cj[u_\de] $ in powers of $ \de $.
\end{proof}
The next corollary is a consequence of the Pohozaev's identities and the fact that our waves are minimizers\footnote{It is possible that the conclusions of Corollary \ref{cor:2} are valid, by just assuming that $\phi$ satisfies the elliptic profile equations, without being a constrained minimizer, but we leave this open at the present time} .

\begin{corollary}
	\label{cor:2}
	Let $\phi_\la$ be a minimizer for either one of \eqref{minI}, \eqref{minIp}, \eqref{minJ}, \eqref{minJp}.
	Then, for each $\la>0$,
	$ \om < 2 $.
\end{corollary}

\begin{proof}
	Let $ \phi_\la $ be  a minimizer for \eqref{minI}, so in particular $\|\phi_\la'\|_{L^2}^2=\la$. Then,  we have
	$m(\la)<\la$, as established in the proof of Lemma \ref{nonv}.  Therefore
	$$
		I(\phi_\la)=\f{1}{2}\int_\rone \left| \phi_\la '' \right|^2 +|\phi_\la|^2dx - \frac{1}{p+1}\int_{\rone} |\phi_\la'|^p\phi_\la'dx<\la= \int_{\rone} |\phi_\la'|^2dx
	$$
	Rearranging terms yields
	\begin{equation}\label{om_leq_2_1}
		\frac{1}{2}\int_\rone \left| \phi_\la '' \right|^2 +|\phi_\la|^2dx  < \int_\rone |\phi_\la'|^2dx+\frac{1}{p+1}\int_{\rone} |\phi_\la'|^p\phi_\la'dx.
	\end{equation}
	Since $ \phi_\la $ also satisfies \eqref{120}, we get
	\begin{equation}\label{om_leq_2_2}
		\frac{1}{2}\int_\rone |\phi_\la''|^2  + |\phi_\la|^2dx  = \frac{\om}{2}
		\int_\rone |\phi_\la'|^2dx + \frac{1}{2}\int_\rone |\phi_\la'|^p\phi_\la'dx.
	\end{equation}
	Combining \eqref{om_leq_2_1} and \eqref{om_leq_2_2}, we have that
	$$
		\left(\f{\om }{2}-1\right) \int_\rone |\phi_\la'|^2dx =-\f{p-1}{2(p+1)} \int_\rone |\phi_\la'|^p\phi_\la'dx.
	$$
	Recalling again that $\int_\rone |\phi_\la'|^p\phi_\la'dx>0$,
	we conclude that $ \om <2 $. 	Similarly for the minimizers of the other three variational problems.
\end{proof}

\section{Stability of the waves}  
\label{sec:4} 
  We start this section by developing the necessary general tools, which imply stability of the waves under consideration. Most of the content in Section \ref{sec:4.1} is contained in some recent papers, but we summarize it for the reader's convenience. Note that our reasoning applies equally well to both the fourth order and the second order versions of the eigenvalue problems, see \eqref{318}, \eqref{319}. 
\subsection{Instability index theory}
\label{sec:4.1} 
In this section, we present the instability index count theory, specifically applied to eigenvalue problems in the form \eqref{319}. As we have previously discussed, the main issue is whether or nor, the eigenvalue problem
\eqref{319} has non-trivial solutions $(\mu, z)$. To that end, we mostly follow the theory  developed in \cite{LZ}, in the specific case when the self-adjoint operator is $\cj=\p_x$.  More specifically,  consider a (slightly more general)  eigenvalue problem
\begin{equation}
\label{325}
\p_x \cl z=\mu z,
\end{equation}
where we  requite the following - there is  Hilbert space $\cx$ over the reals, so that
\begin{itemize}
	\item $\cl: \cx\to \cx^*$ is a bounded and symmetric operator, in the sense that $(u,v)\to \dpr{\cl u}{v}$ is  a bounded symmetric form on $\cx\times \cx$.
	\item $dim(Ker[\cl])<\infty$ and moreover, there is an $\cl$ invariant decomposition
	$$
	\cx=\cx_-\oplus Ker[\cl]\oplus \cx_+, dim(\cx_-)<\infty,
	$$
	so that for some $\de>0$, $\cl_-|_{\cx_-}\leq -\de$,  $\cl_+|_{\cx_+}\geq \de$. That is,  for every $u_\pm \in \cx_\pm$, there is  $\dpr{\cl u_-}{u_-}\leq -\de \|u_-\|_{\cx_-}^2$ and $\dpr{\cl u_+}{u_+}\geq \de \|u_+\|_{\cx_+}^2$.
\end{itemize}
Introduce the Morse index $n^-(\cl):=dim(\cx_-)$, which is equivalent to the number of negative eigenvalues of the  operator $\cl$, counted with their respective multiplicities.  Consider the generalized eigenspace $E_0=\{u\in \cx: (\p_x \cl)^k u=0,k=1,2,\ldots\}$. Clearly $Ker[\cl]\subset E_0$, so consider the complement in $E_0$ of $Ker[\cl]$. That is,
$E_0=Ker[\cl]\oplus \tilde{E}_0$. Let
$$
k_0^{\leq 0}:=\max\{dim(Z): Z\  \textup{subspace of}\  \tilde{E}_0: \dpr{\cl z}{z}< 0, z\in Z\}.
$$
Theorem 2.3, \cite{LZ}) asserts that\footnote{ Theorem 2.3, \cite{LZ} is actually much more general,
	but we state this corollary, as it is enough for us} the number of solutions of \eqref{325}, $k_{unstable}$ is estimated by
\begin{equation}
\label{a:40}
k_{unstable}\leq n^-(\cl)- k_0^{\leq 0}(\cl).
\end{equation}
In particular, and this is what we use below, if $n^-(\cl)=1$ and $k_0^{\leq 0}(\cl)\geq 1$, the   problem \eqref{325} is spectrally  stable. At this point, we point out that the spectral problem \eqref{319} conforms to this framework, once we take the Hilbert space $\cx:=H^1(\rone)\cap \dot{H}^{-1}(\rone)$ over the reals.

Let us now derive the so-called Vakhitov-Kolokolov criteria for stability\footnote{Although the original criteria and his derivation was done, strictly speaking in the NLS context, it introduces an important  quantity, which turns out to be  relevant in wide class of  Hamiltonian stability problems.}. Assume that $\Psi$ is sufficiently smooth, $\Psi'\in Ker[\cl]$ and in addition, assume $\Psi\perp Ker[\cl]$. Then, we can identify $Q:=\cl^{-1}[\Psi]$ as element of $Ker[(\p_x \cl)^2]\setminus Ker[(\p_x \cl)]\subset \tilde{E}_0$. Indeed,
$\p_x \cl Q=\Psi'$, while $(\p_x \cl)^2 Q=\p_x \cl \Psi'=\p_x \cl \Psi'=0$.  Now, if $\dpr{\cl Q}{Q}<0$, we clearly can conclude that
$k_0^{\leq 0}(\cl)\geq 1$ (which together with $n^-(\cl)=1$ would imply stability by \eqref{a:40}). On the other hand,
$\dpr{\cl Q}{Q}=\dpr{\cl \cl^{-1} \Psi}{\cl^{-1} \Psi}=\dpr{\cl^{-1} \Psi}{\Psi}$. Note that in our specific case, the eigenvalue problem \eqref{319} satisfies $\cl[\phi'_\la]=0$, as long as $\phi_\la$ is a minimizer of \eqref{41}, \eqref{51} respectively.
Thus, we have proved the following
\begin{corollary}
	\label{cor:lo98}
	Suppose that the wave $\phi_\la$ satisfies
	\begin{enumerate}
		\item $n^-(\cl_+)=1$
		\item the wave $\phi_\la$ is weakly non-degenerate,  
		\item $\dpr{\cl_+^{-1} \phi_\la}{\phi_\la}<0$.
	\end{enumerate}
	Then, the wave is strongly spectrally stable, in the sense that the eigenvalue problem \eqref{319} does not have non-trivial solutions, and in fact $\si(\p_x \cl_+)\subset i \rone$.
\end{corollary}
\noindent This corollary is our main tool for establishing strong spectral stability for our waves $\phi_\la$.

Another useful observation is 
\begin{proposition}
	\label{prop:lk} 
	Assume weak non-degeneracy, i.e. $\phi\perp Ker[\cl_+]$. Then, the condition \eqref{nond}, that is $\dpr{\cl_+^{-1} \phi}{\phi}\neq 0$ guarantees that 
	the zero eigenvalue for $\p_x \cl$ has algebraic multiplicity two. 
\end{proposition}
\begin{proof}
	We have already discussed that the zero eigenvalue has algebraic multiplicity  two, since $\p_x \cl_+[\phi']=0$, $(\p_x \cl_+)^2  [\cl_+^{-1} \phi]= \p_x \cl_+[\phi']=0$. Assume that there is a third element in the chain, say $\zeta$. It must be that 
	$\p_x \cl_+ \zeta=\cl_+^{-1} \phi$. This last equation is not solvable, unless $\cl_+^{-1} \phi=q'$ for some $q$. Assuming that, we must have $\cl_+ \zeta =q$. Taking dot product with $\phi'$, it follows that $q\perp \phi'$. Then, 
	$$
	0=\dpr{q}{\phi'}=-\dpr{q'}{\phi}=-\dpr{\cl_+^{-1} \phi}{\phi},
	$$
	a contradiction with \eqref{nond}. 
\end{proof}

\subsection{Weak non-degeneracy of the waves} 
Our first order of business is to show that $\phi_\la\perp Ker[\cl_+]$. Let us work with the second order version, for which $\cl_+=-\p_x^2-\p_x^{-2}- \om-p|\phi_\la|^{p-1}$, the other one being similar.  Take any element $\Psi\in Ker[\cl_+], \|\Psi\|_{L^2}=1$.  Note that by Proposition \ref{prop:30}, we have that $\cl_+|_{\{\phi_\la\}^\perp}\geq 0$. It follows that 
$\Psi-\la^{-1} \dpr{\Psi}{\phi_\la} \phi_\la\perp \phi_\la$, since by construction $\|\phi_\la\|^2=\la$. Thus, 
\begin{eqnarray*}
0\leq \dpr{\cl_+[\Psi-\la^{-1} \dpr{\Psi}{\phi_\la} \phi_\la]}{\Psi-\la^{-1} \dpr{\Psi}{\phi_\la} \phi_\la}=\la^{-2} \dpr{\Psi}{\phi_\la}^2 \dpr{\cl_+ \phi_\la}{\phi_\la}\leq 0,
\end{eqnarray*} 
where in the last inequality we have used \eqref{lplus}. We conclude that $\dpr{\Psi}{\phi_\la}=0$, otherwise we have a contradiction in the above chain of inequalities.  

\subsection{General spectral stability lemma}
We have the following general result. 
\begin{lemma}
	\label{stab:10} 
	Let $\ch$ be a self-adjoint operator on a Hilbert space $H$, so that $\ch|_{\{\xi_0\}^\perp}\geq 0$ for some vector $\xi_0: \xi_0\perp Ker[\ch], \|\xi_0\|=1$. Assume that $\dpr{\ch \xi_0}{\xi_0}\leq 0$. Then, 
	$$
	\dpr{\ch^{-1} \xi_0}{\xi_0}\leq 0.
	$$  
\end{lemma}
{\bf Remark:} Note that the condition $\xi_0\perp Ker[\ch]$ guarantees that $\ch^{-1} \xi_0$ is well-defined. 
\begin{proof}
Due to the self-adjointness, $\ch^{-1}: Ker[\ch]^\perp \to Ker[\ch]^\perp$. Consider 
$$
\eta_0:=\ch^{-1}\xi_0 - \dpr{\ch^{-1} \xi_0}{\xi_0} \xi_0\perp \xi_0.
$$
 Note  that $\dpr{\ch^{-1} \xi_0}{\xi_0}$ is a real. Thus, 
\begin{eqnarray*}
	0 &\leq & \dpr{\ch\eta_0}{\eta_0}=\dpr{\ch[\ch^{-1}\xi_0 - \dpr{\ch^{-1} \xi_0}{\xi_0} \xi_0]}{\ch^{-1}\xi_0 - \dpr{\ch^{-1} \xi_0}{\xi_0} \xi_0}=\\
	&=& \dpr{ \xi_0 - \dpr{\ch^{-1} \xi_0}{\xi_0} \ch \xi_0]}{\ch^{-1}\xi_0 - 
		\dpr{\ch^{-1} \xi_0}{\xi_0} \xi_0}= \\
	&=& -\dpr{\ch^{-1} \xi_0}{\xi_0} + \dpr{\ch^{-1} \xi_0}{\xi_0}^2 \dpr{\ch \xi_0}{\xi_0} \leq -\dpr{\ch^{-1} \xi_0}{\xi_0}. 
\end{eqnarray*} 
	Thus, $\dpr{\ch^{-1} \xi_0}{\xi_0}\leq 0$.

\end{proof}

\subsection{The proof of Theorem \ref{theo:10}} 
The proof of Theorem \ref{theo:10} consists of applying Lemma \ref{stab:10} to $\ch=\cl_+$ and $\xi_0:=\la^{-1/2} \phi_\la$. We have shown that $\cl_+|_{\{\phi_\la\}^\perp}\geq 0$ and \eqref{lplus} establishes $\dpr{\cl_+  \phi_\la}{\phi_\la}\leq 0$.  This verifies all the assumptions in Lemma \ref{stab:10}, which implies $\dpr{\cl_+^{-1} \phi}{\phi}\leq 0$.   Assuming \eqref{nond}, means that $\dpr{\cl_+^{-1} \phi}{\phi}<0$, which then feeds into Corollary \ref{cor:lo98} to imply that the stability.


\begin{thebibliography}{99}
	
	
\bibitem{GP} A. Geyer,  D. Pelinovsky, {\emph Spectral stability of periodic waves in the generalized reduced Ostrovsky equation.}, {\em Lett. Math. Phys.},{\bf 107},  (2017), no. 7,  p. 
1293--1314.

	\bibitem{GOS} R. Grimshaw,  L. Ostrovsky,  V. Shrira, Y. Stepanyants, {\emph Long non-
			linear surface and internal gravity waves in a rotating ocean.} {\em Surveys Geophys.} {\bf  19}, (1998),
	p. 289--338.

	\bibitem{GL} G. Gui, Y. Liu, {\emph  On the Cauchy problem for the Ostrovsky equation with positive dispersion},
	{\em  Comm. Partial Differential Equations}, {\bf 32},  (2007), no. 10-12, p. 1895--1916.


 

\bibitem{JP} E. Johnson,  D. Pelinovsky, {\emph Orbital stability of periodic waves in the class of reduced Ostrovsky equations.}, {\em J. Differential Equations}, {\bf 261},  (2016), no. 6,  p. 
3268--3304. 



	\bibitem{L12} S. Levandosky, {\emph On the stability of solitary waves of a generalized Ostrovsky equation}, {\em Anal. Math. Phys.} {\bf 2},  (2012), no. 4,  p. 407--437.

	\bibitem{LL6}   S.   Levandosky, Y. Liu,  {\emph  Stability of solitary waves of a generalized Ostrovsky equation. }, {\em SIAM J. Math. Anal.}, {\bf 38},  (2006), no. 3,  p. 985--1011.

	\bibitem{LL7} S.   Levandosky, Y. Liu,  {\emph Stability and weak rotation limit of solitary waves of the Ostrovsky equation.} {\em Discrete Contin. Dyn. Syst. Ser. B}, {\bf 7},  (2007), no. 4,  p. 793--806.

	\bibitem{LZ} Z. Lin, C.  Zeng, {\emph Instability, index theorem, and exponential trichotomy for Linear Hamiltonian PDEs}, available at
	https://arxiv.org/abs/1703.04016


	\bibitem{LO7}  Y. Liu, {\emph On the stability of solitary waves for the Ostrovsky equation},  {\em
			Quart. Appl. Math.} {\bf 65},  (2007), no. 3,  p. 571--589.
		
\bibitem{LPS} 		Y. Liu, D. Pelinovsky, A. Sakovich, {\emph Wave breaking in the Ostrovsky-Hunter equation}, {\em SIAM J. Math. Anal.}, {\bf 42},  (2010), p. 1967--1985.

	\bibitem{LO8}  Y. Liu, M.  Ohta, {\emph Stability of solitary waves for the Ostrovsky equation},
	{\em Proc. Amer. Math. Soc.} {\bf 136},  (2008), no. 2, p. 511--517.

	\bibitem{LV04} Y. Liu, V.  Varlamov, {\emph Stability of solitary waves and weak rotation limit for the Ostrovsky equation},  {\em J. Differential Equations}, {\bf 203},  (2004), no. 1,  p. 159--183.

	\bibitem{LM} F. Linares, A. Milane\'es, {\emph Local and global well-posedness for the Ostrovsky equation},
	{\em J. Differ. Equ.} {\bf  222}, (2006),  p. 325--340.

	\bibitem{O78} L. Ostrovsky,   {\emph Nonlinear internal waves in a rotating ocean}, {\em Okeanologia}, {\bf 18}, (1978),  p. 181--191.
	\bibitem{OS90} L. Ostrovsky,  Y. Stepanyants, {\emph Nonlinear surface and internal waves in rotating fluids},
	{\em Research Reports in Physics, Nonlinear Waves}, {\bf 3}, Springer, Berlin (1990)

\bibitem{Psak} D. Pelinovsky, A. Sakovich, {\emph Global well-posedness of the short-pulse and sine-Gordon equations in energy space}, {\em Comm. Part. Diff. Eqs.}, {\bf 35},  (2010), 
p. 613--629.

	\bibitem{SW} T. Sch\"afer,  C. Wayne, {\em Propagation of ultra-short optical pulses in cubic nonlinear media.}
		{\em Phys. D}, {\bf 196},  (2004), no. 1-2, p. 90--105.

	\bibitem{S1} A. Stefanov, {\emph On the normalized ground states of second order PDE's with mixed power non-linearities}, submitted.

	\bibitem{Ts} K. Tsugawa, {\emph Well-posedness and weak rotation
			limit for the Ostrovsky equation.},  {\em J. Differential Equations}, {\bf 247},
	(2009), no. 12,  p. 3163--3180.


	\bibitem{VL} V. Varlamov,  Y.  Liu, {\emph Cauchy problem for the Ostrovsky equation.} {\em Discrete Dyn. Syst.} , {\bf 10}, (2004), p. 731--751.

	\bibitem{ZL} P. Zhang,  Y.  Liu, {\emph Symmetry and uniqueness of the solitary-wave solution for the Ostrovsky equation.} {\em Arch. Ration. Mech. Anal.} {\bf 196},  (2010), no. 3,  p. 811--837.





\end{thebibliography}
\end{document}